\documentclass[12pt,reqno]{amsart}%
\usepackage{amsmath, amsfonts, amssymb, amsthm, amscd, amsbsy}
\usepackage{fancyhdr}
\usepackage[usenames,dvipsnames,svgnames,x11names,hyperref]{xcolor}
\usepackage{geometry}
\usepackage{graphicx}
\usepackage{hyperref}
\usepackage{amsmath}
\usepackage{amsfonts}
\usepackage{amssymb}
\usepackage{lineno}%
\providecommand{\U}[1]{\protect\rule{.1in}{.1in}}
\hypersetup{colorlinks,breaklinks,
	linkcolor={Fuchsia},
	citecolor={ForestGreen},
	urlcolor={NavyBlue}}
\newtheorem{theorem}{Theorem}[section]
\theoremstyle{plain}

\newtheorem{corollary}{Corollary}[section]

\newtheorem{proposition}{Proposition}[section]

\numberwithin{equation}{section}
\allowdisplaybreaks

\newcommand{\mr}{\mathbb{R}}
\newcommand{\p}{\partial}
\newcommand{\ent}{\text{Ent}_{d\mu_w}}

\begin{document}
\title[Beckner, Poincar\'e and log Sobolev inequalities]{Log-Sobolev and Beckner inequalities and stability of Poincar\'e inequality with weighted Gaussian measures}
\author{Nguyen Lam}
\address{Nguyen Lam: School of Science and the Environment, Grenfell Campus, Memorial
University of Newfoundland, Corner Brook, NL A2H5G4, Canada}
\email{nlam@mun.ca}
\author{Guozhen Lu}
\address{Guozhen Lu: Department of Mathematics, University of Connecticut, Storrs, CT
06269, USA}
\email{guozhen.lu@uconn.edu}
\author{Andrey Russanov}
\address{Andrey Russanov: Department of Mathematics, University of Connecticut, Storrs,
CT 06269, USA}
\email{andrey.russanov@uconn.edu}

\begin{abstract}
We employ a Markov semigroup approach combined with the \( \Gamma \)-calculus to establish a generalized Beckner inequality associated with weighted Gaussian measures. As a direct consequence, we derive the corresponding Poincaré inequality in the same setting. Subsequently, by means of a duality argument, we investigate gradient and \(L^2\) stability estimates of the Poincaré inequality. Furthermore, we formulate a scale-dependent version of the Poincaré inequality for homogeneous Gaussian-type measures and apply it to analyze the stability of the Heisenberg Uncertainty Principle with homogeneous weights. Finally, we establish a Logarithmic Sobolev inequality for weighted Gaussian measures and utilize it to derive the Euclidean Logarithmic Sobolev inequality with homogeneous log-concave weights.

\end{abstract}
\maketitle

\section{Introduction}

Let \( w : \mathbb{R}^n \to [0, \infty) \) be a nonnegative function, with support given by the set
\[
\Sigma := \{ x \in \mathbb{R}^n : w(x) > 0 \}.
\]
We impose a mild convexity condition on \( w \) given by
\[
 -\nabla^2 (\log w(x)) \ge K\, \mathbb{I}
 \quad \text{for all } x \in \Sigma,\text{for some constant \( K > -1 \)}, \tag{$\star$}\label{star}
\]
where \( \mathbb{I} \) denotes the identity matrix. The largest constant \( K \) for which this inequality holds is denoted by \( K_w \). This structural assumption is particularly important for establishing the \emph{curvature-dimension (CD)} condition associated with the weighted Gaussian measure \( w(x)e^{-|x|^2/2}\,dx \).

Sometimes we also require that \( w \) is \emph{homogeneous of degree} \( \alpha \ge 0 \). That is, for every \( t > 0 \) and all \( x \in \Sigma \), the scaling relation $ w(tx) = t^{\alpha} w(x)$ holds. In this case, for simplicity, we assume that $\Sigma$ is an open convex cone in $\mathbb{R}^n$, with vertex at the origin such that $x \cdot \mathbf{\eta}=0$ on $\partial\Sigma$, where $\eta$ is the outward pointing unit normal of surface
elements $dS$. Typical examples include:
\begin{itemize}
  \item \emph{Monomial weights:} \( w(x) = x^A := |x_1|^{\alpha_1} \cdots |x_n|^{\alpha_n} \), with \( \alpha = \alpha_1 + \cdots + \alpha_n \);
  \item \emph{Dunkl-type weights:} \( w(x) = \prod_{\beta \in R_+} |\langle \beta, x \rangle|^{2k_\beta} \), where \( R_+ \) is a positive root subsystem and \( \alpha = 2\sum_{\beta \in R_+} k_\beta \);
  \item \emph{Radial weights:} \( w(x) = |x|^\alpha \). 
\end{itemize}

By Euler's homogeneous function theorem, every such homogeneous weight satisfies the differential identity
\[
 x \cdot \nabla w(x) = \alpha\, w(x),
\]
which encodes its degree of homogeneity in analytical form. 
We note that if $w$ is homogeneous of degree $\alpha \geq 0$ and satisfies \eqref{star}, then the curvature parameter satisfies \( K_w \leq 0 \). In particular, a homogeneous log-concave function $w$ must have $K_w=0$. Indeed, assume that $-\nabla^2 (\log w(x)) \ge K\, \mathbb{I}$ for some $K>0$. Then define $u(x)=\log w(x)$. Homogeneity of \(w\) then implies $u(tx)=\alpha\log t + u(x)$, so that, in particular, the Euler identity in logarithmic form is
\begin{equation*}
	x\cdot\nabla u(x)=\alpha.
\end{equation*}
Fix \(x\in\Sigma\) and consider the ray \(t\mapsto tx\) for \(t>0\). Define $	\varphi(t)=u(tx),~t>0$. From the functional identity for \(u\),
\begin{equation*}
	\varphi(t)=\alpha\log t + u(x).
\end{equation*}
Differentiate twice in \(t\):
\begin{equation*}
	\varphi'(t)=\frac{\alpha}{t},\qquad
	\varphi''(t)=-\frac{\alpha}{t^2}.
\end{equation*}
On the other hand, by the chain rule,
\begin{equation*}
	\varphi'(t)=x^T\nabla u(tx),\qquad
	\varphi''(t)=x^T \nabla^2 u(tx)x.
\end{equation*}
Thus
\begin{equation*}
	x^T \nabla^2 u(tx)x = -\frac{\alpha}{t^2}.
\end{equation*}
In particular, at \(t=1\),
\begin{equation*}
	x^T \nabla^2 u(x)x = -\alpha.
\end{equation*}
So the radial quadratic form is the constant \(-\alpha\), independent of the size of \(x\).
Now if
\begin{equation*}
	-\nabla^2(\log w)(x) = -\nabla^2 u(x)\ge K\, \mathbb{I}
\end{equation*}
for all \(x\in\Sigma\), then for every \(x\in\Sigma\),
\begin{equation*}
	x^T\bigl(-\nabla^2 u(x)\bigr)x\ge Kx^T \mathbb{I} x = K|x|^2,
\end{equation*}
or equivalently
\begin{equation*}
	x^T \nabla^2 u(x)x\le -K|x|^2.
\end{equation*}
But from the radial identity above,
\begin{equation*}
	x^T \nabla^2 u(x)x = -\alpha.
\end{equation*}
Combining the two relations yields
\begin{equation*}
	-\alpha\le -K|x|^2,
\end{equation*}
which is impossible if $K>0$, since \(|x|\) can be made arbitrarily large in \(\Sigma\). 

Finally, we call a weight \( w \) \emph{partial} if it is independent of at least one coordinate, i.e., \( \partial_{x_i} w \equiv 0 \) for some index \( i \).  

In this paper, we study Poincar\'e, logarithmic Sobolev, and Beckner inequalities for weighted Gaussian measures of the form
\begin{equation*}
    d\mu_w = \frac{w(x) e^{-|x|^2/2} \, dx}{\int_\Sigma w(x) e^{-|x|^2/2} \, dx},
\end{equation*}
along with their applications. We pursue three primary goals. First, we establish a generalized Beckner inequality on $(\Sigma, \mu_w)$. Second, we derive a Poincar\'e inequality on $(\Sigma, \mu_w)$, including its gradient stability, and apply them to establish stability estimates for the Heisenberg uncertainty principle under homogeneous weights. Third, we investigate a logarithmic Sobolev inequality on $(\Sigma, \mu_w)$, using it to obtain a Euclidean logarithmic Sobolev inequality with homogeneous weights and explicit constants. 

\medskip

Our first main goal of this paper is to establish a complete family of sharp Beckner-type inequalities for weighted Gaussian measures on conical domains $\Sigma$, extending the classical Gaussian results to weighted settings controlled by the curvature-dimension parameter $K_w>-1$. Originally proved by Beckner \cite{Bec89} via explicit Hermite expansions, the Beckner inequality provides a powerful interpolation between Poincar\'e and logarithmic Sobolev inequalities:
\begin{equation*}
\frac{1}{2-p} \left( \int_{\mr^n} f^2\,d\mu - \Bigl( \int_{\mr^n} |f|^p\,d\mu \Bigr)^{2/p} \right) 
\le C \int_{\mr^n} |\nabla f|^2\,d\mu, \quad 1\le p<2.
\end{equation*}
Here $d\mu$ denotes the classical Gaussian measure: 
$$d\mu = \frac{e^{-\frac{|x|^2}{2}}}{(2\pi)^{n/2}}dx.$$
As $p\to 1^+$, this recovers the Poincar\'e spectral gap; as $p\to 2^-$, it yields the log-Sobolev inequality via the expansion $|f|^p = e^{p\log|f|} \approx 1 + p\log|f|$. These inequalities have profound significance across analysis, probability, and geometry. For instance, they underpin hypercontractivity of Ornstein-Uhlenbeck semigroups \cite{GRO75b, Gross06, Nelson73}, enabling exponential concentration and moment bounds; control uncertainty principles in Fourier analysis \cite{Beckner75}; and yield sharp functional inequalities on manifolds with non-negative Ricci curvature \cite{CZ25, Ledoux00}. See, e.g., Bakry, Gentil, and Ledoux~\cite{BGL14} for their deep connections to the curvature-dimension condition.

Motivated by the results in \cite{Gwynne13}, we will use the semigroup approach to establish the following generalized Beckner inequality which is our first main result in this paper: 

\begin{theorem}\label{T1}
Let $1\le p<q$. Assume that $w$ satisfies \eqref{star}. Then for  $W^{1,q}_A(\Sigma,\mu_w)$, we have 
\begin{equation}
\frac{1}{q-p}\Biggl[ \Bigl( \int_\Sigma |f|^q\,d\mu_w \Bigr)^{2/q} - \Bigl( \int_\Sigma |f|^p\,d\mu_w \Bigr)^{2/p} \Biggr] 
\le \frac{1}{1+K_w} \Bigl( \int_\Sigma |\nabla f|^q\,d\mu_w \Bigr)^{2/q}.
\end{equation}

\end{theorem}

Here we denote the Sobolev space $W^{1,q}_A(\Sigma,\mu_w)$ as the completion with respect to the norm
\[
\left( \int_\Sigma |f|^q \, d\mu_w \right)^{1/q} + \left( \int_\Sigma |\nabla f|^q \, d\mu_w \right)^{1/q}
\]
of $A := \left\{u \in C_0^\infty(\overline{\Sigma}) \ \big | \  \nabla u \cdot \eta = 0~\text{on}~ \partial\Sigma\right\}$, where $\eta$ is the outer normal vector of $\Sigma$.

In particular, when $q=2$, we obtain the following Beckner-type inequality for weighted Gaussian measures:

\begin{theorem}\label{T2}
Assume that $w$ satisfies \eqref{star}. For $1\le p<2$ and $f\in W^{1,2}_A(\Sigma,\mu_w)$ 
\begin{equation*}
\frac{1}{2-p} \left( \int_\Sigma |f|^2\,d\mu_w - \Bigl( \int_\Sigma |f|^p\,d\mu_w \Bigr)^{2/p} \right) 
\leq \frac{1}{1+K_w} \int_\Sigma |\nabla f|^2\,d\mu_w.
\end{equation*}
The constant $\frac{1}{1+K_w}$ is sharp  when $w$ is partial.
\end{theorem}

It is known that one can deduce the Poincar\'e inequality from the Beckner inequality. Therefore, as a consequence of our main result (Theorem \ref{T1}), by Taylor expansion, we obtain the following Poincar\'e inequality for weighted Gaussian measures:

\begin{theorem}\label{T2.1}
Assume that $w$ satisfies \eqref{star}. For $ q\geq 2$ and $f\in W^{1,q}_A(\Sigma,\mu_w)$, we have 
    $$\int_\Sigma\left\vert f-\int_\Sigma
        fd\mu_{w}\right\vert ^{2}d\mu_{w} \leq \frac{1}{1+K_w}\left(\int_\Sigma |\nabla f|^q\,d\mu_w\right)^{\frac{2}{q}}.$$
\end{theorem}

Our next primary goal is to establish  some sharp stability estimates for the Poincar\'e inequality on the weighted Gaussian measure $\mu_w$, along with applications. We will prove that

\begin{theorem}
\label{T3}
    Assume that $w$ satisfies \eqref{star}. For all $f \in W^{1,2}_A(\Sigma,\mu_w)$, we have
    \begin{equation}
        \int_\Sigma\left\vert \nabla f\right\vert ^{2}d\mu_{w}
        \geq \left(1+K_w\right)\int_\Sigma\left\vert f-\int_\Sigma
        fd\mu_{w}\right\vert ^{2}d\mu_{w}. \label{PGH}%
    \end{equation}
    If $w$ is partial, (\ref{PGH}) can be attained by  non-constant functions. 
    
    Moreover, \begin{align}
    &\int_\Sigma|\nabla f|^2\,d\mu_w - \left(1+K_w\right)\int_\Sigma \Bigl| f - \int_\Sigma f\,d\mu_w \Bigr|^2 d\mu_w \nonumber\\
    &\geq \frac{1}{2} \int_\Sigma \Biggl| \nabla \Biggl[ f - \left(1+K_w\right)\int_\Sigma \Bigl( f - \int_\Sigma f\,d\mu_w \Bigr) x\,d\mu_w \cdot x \Biggr] \Biggr|^2 d\mu_w. \label{IPH}
\end{align}
Also, if $w$ is partial, equality in \eqref{IPH} is attained by non-affine functions.
\end{theorem}

We can prove \eqref{PGH} in Theorem \ref{T3} by using the $\Gamma$-calculus framework of Bakry and \'Emery \cite{BE85}. Here, we will provide an alternative proof via the duality-based $L^2$ H\"ormander method (see \cite{Bonn22, FIL16}, for instance). A refined analysis of the this approach yields improvements to the Poincar\'e inequality~\eqref{T1}, including its gradient stability estimate. As a consequence, we also obtain the following $L^2$ stability of the Poincar\'e inequality. 

\begin{corollary}
Assume that $w$ satisfies \eqref{star}. For $f\in W^{1,2}_A(\Sigma,\mu_w)$, we have
\begin{align}      
    &\int_\Sigma |\nabla f|^2\,d\mu_w - \left(1+K_w\right)\int_\Sigma \left\vert f -\int_\Sigma f\,d\mu_w\right\vert^2 d\mu_w \nonumber\\
    &\geq \frac{1+K_w}{2} \int_\Sigma \Bigg| f - \int_\Sigma f\,d\mu_w - \left(1+K_w\right)\int_\Sigma fx\,d\mu_w\cdot x + (1+K_w) \left(\int_\Sigma f\,d\mu_w\right)\left(\int_\Sigma x\,d\mu_w\right)\cdot x \nonumber \\ 
    &- (1+K_w)\left(\int_\Sigma f\,d\mu_w\right) \left|\int_\Sigma x\,d\mu_w\right|^2+ \left(1+K_w\right)\Bigl( \int_\Sigma fx\,d\mu_w \Bigr) \cdot \Bigl( \int_\Sigma x\,d\mu_w \Bigr) \Bigg|^2 d\mu_w. \label{C0}
\end{align}

\end{corollary}

The above results provide a gradient stability refinement \eqref{IPH} and a $L^2$-stability version \eqref{C0} of the Poincar\'e inequality \eqref{T1}. In the unweighted case $w=1$, they recover the classical Gaussian Poincar\'e inequality with its optimal stability estimates.

The stability of geometric and functional inequalities was first explored by Brezis and Lieb in \cite{BL85}, who questioned whether the deficit in Sobolev inequalities bounds the distance to extremizers. Bianchi and Egnell affirmatively resolved this in \cite{BE91} by exploiting the structure of $W^{1,2}(\mathbb{R}^{n})$, proving
\[
\int_{\mathbb{R}^{n}}|\nabla f|^2\,dx - S_{n}\left( \int_{\mathbb{R}^{n}}|f|^{\frac{2n}{n-2}}\,dx \right)^{\frac{n-2}{n}} \geq c_{BE} \inf_{F \in E_{\mathrm{Sob}}} \int_{\mathbb{R}^{n}} |\nabla (f - F)|^2 \, dx
\]
for some $c_{BE} > 0$, where $S_n$ is the optimal Sobolev constant and $E_{\mathrm{Sob}}$ the optimizer manifold. This initiated extensive research on quantitative stability; see, e.g., \cite{BWW03,BDNN20,CF13,CFW13,CLT23,CLTW2025,CFMP09,FJ1,FJ2,FN19,FZ22,LW99} and references therein for Sobolev-type inequalities.

Stability constants and their attainability have often been overlooked. Recently, Dolbeault, Esteban, Figalli, Frank, and Loss in \cite{DEFFL} provided asymptotically optimal lower bounds for $c_{BE}$ as $n \to \infty$ using a gradient flow to globalize local Sobolev stability, with applications to Gaussian log-Sobolev. K\"onig showed in \cite{Kon23} that the optimal lower bound is strictly below the spectral gap $\frac{4}{n+4}$ and proved attainability in \cite{Kon22}. Chen, Tang, and the second author obtained explicit lower bounds for Hardy-Littlewood-Sobolev stability in \cite{CLT24}, extending to higher/fractional Sobolev, with asymptotic optimality as $n \to \infty$ ($0 < s < n/2$) or $s \to 0$ in \cite{CLT242,CLT243}; the latter yields global sphere log-Sobolev stability, refining the work of \cite{CLT23} and Beckner \cite{Beckner93}. Most recently, \cite{CLTW2025} establishes optimal Sobolev stability on the Heisenberg group via CR Yamabe flow, bypassing absent P\'olya-Szeg\"{o} and Riesz rearrangement inequalities.

In recent works such as \cite{CFLL24,LLR25}, classical Gaussian Poincar\'e inequalities and their stability versions have been used to prove sharp stability results for the classical Heisenberg uncertainty principle (HUP) and its refinements. The HUP originates in quantum mechanics and admits the following mathematical formulation: for $f \in S_0$,
\begin{equation}
\left(  \int_{\mathbb{R}^{n}}\left\vert \nabla f\right\vert ^{2}dx\right)
\left(  \int_{\mathbb{R}^{n}}|f|^{2}|x|^{2}dx\right)  \geq\frac{n^{2}}
{4}\left(  \int_{\mathbb{R}^{n}}|f|^{2}dx\right)  ^{2}. \label{HUP}%
\end{equation}
Here, $S_{0}$ is the completion of $C_{0}^{\infty}\left(  \mathbb{R}
^{n}\right)  $ under the norm $\left(  \int_{\mathbb{R}^{n}}\left\vert \nabla
f\right\vert ^{2}dx\right)  ^{\frac{1}{2}}+\left(  \int_{\mathbb{R}^{n}
}|f|^{2}|x|^{2}dx\right)^{\frac{1}{2}}$. In \cite{CFLL24}, the authors combined the Gaussian Poincar\'e inequality with the HUP identity: for $f\in S_{0}$, $f\neq0$ and
$\lambda=\left(  \frac{\int_{\mathbb{R}^{n}}|f|^{2}|x|^{2}dx}{\int
_{\mathbb{R}^{n}}|\nabla f|^{2}dx}\right)  ^{\frac{1}{4}}$, one has
\begin{equation}
\left(  \int_{\mathbb{R}^{n}}\left\vert \nabla f\right\vert ^{2}dx\right)
^{\frac{1}{2}}\left(  \int_{\mathbb{R}^{n}}|f|^{2}|x|^{2}dx\right)  ^{\frac
{1}{2}}-\frac{n}{2}\int_{\mathbb{R}^{n}}|f|^{2}dx=\frac{\lambda^{2}}{2}
\int_{\mathbb{R}^{n}}\left\vert \nabla\left(  fe^{\frac{|x|^{2}}{2\lambda^{2}
}}\right)  \right\vert ^{2}e^{-\frac{|x|^{2}}{\lambda^{2}}}dx\label{HUPI}%
\end{equation}
to derive sharp HUP stability that sharpens prior results in the literature \cite{F21,MV21}. See also \cite{DLL26, DN26} for the second order case.

Inspired by \cite{CFLL24}, we leverage our Poincar\'e inequalities and their improvements to study sharp HUP stability under homogeneous weights. Specifically, in Section 3, we establish the weighted HUP:
\begin{equation}\label{HUPH}
\left(\int_\Sigma |\nabla f|^2 w(x)dx\right)^{\frac{1}{2}}\left(\int_\Sigma |f|^{2}| x|^2 w(x)dx\right)^{\frac{1}{2}} \geq\frac{n+\alpha}{2} \int_\Sigma |f|^{2}w(x)dx, 
\end{equation}
where the constant $\frac{n + \alpha}{2}$ is optimal and the optimizers are $E_{HUPW} = \left\{ce^{-\beta |x|^2} : c\in \mr , \beta > 0\right\}$. Using our Poincar\'e inequalities (Theorem \ref{T3}), we then prove sharp stability versions, including the following: 
\begin{theorem}
    \label{S1}Assume that $w$ satisfies \eqref{star} and is homogeneous of degree $\alpha \geq 0$. We have
    \begin{equation}
        \delta_w(f) \geq (1+K_w) d^2_w \left(f, E_{HUPW}\right).
    \end{equation}
    Moreover, if $w$ is partial, the equality is achieved by non-trivial functions $u\notin E_{HUPW}$.
\end{theorem}
Here $$d_w (f, E_{HUPW}) = \inf_{c,\lambda\neq 0} \left(\int_\Sigma \left|f - ce^{-\frac{|x|^2}{2\lambda^2}}\right|^2 w(x)dx\right)^{\frac{1}{2}}$$ and
$$
\delta_{w}\left(  f\right)  :=\left(  \int_\Sigma|\nabla
f|^{2}w(x)dx\right)  ^{\frac{1}{2}}\left(  \int_\Sigma|f|^{2}|x|^{2}w(x)dx\right)  ^{\frac{1}{2}}-\frac{n+\alpha}{2}\int_\Sigma|f|^{2}w(x)dx.$$

It is worth mentioning that functional and geometric inequalities with homogeneous weights have attracted substantial interest recently; see, e.g., \cite{DPH25, LLR25, LLZ19, Wan22}. Notably, Bakry--Gentil--Ledoux combined stereographic projection with the curvature-dimension (CD) condition to prove sharp Sobolev inequalities with monomial weights \cite{BGL14}. Prompted by a question of Haim Br\'ezis \cite{Brezis1}, Cabr\'e and Ros-Oton first analyzed regularity of stable solutions to reaction-diffusion equations in double-revolution domains \cite{CR}, later establishing Sobolev, Morrey, Trudinger, and isoperimetric inequalities with monomial weights \cite{CRO}. Optimal Trudinger--Moser and Hardy inequalities followed in \cite{Lam17}. Nguyen then used mass transport to extend these results and identify optimal constants and extremals in Gagliardo--Nirenberg and logarithmic Sobolev inequalities with monomial weights \cite{NVH}. The $p$-Bessel pair framework yielded sharp $L^p$-Hardy and Hardy--Sobolev inequalities with monomial weights \cite{DLL22}, later generalized to Dunkl weights \cite{ADNP}. Meanwhile, Cabr\'e--Ros-Oton--Serra applied the ABP method to show that, for homogeneous weights $w$ of degree $\alpha \ge 0$ with $w^{1/\alpha}$ concave (when $\alpha > 0$), origin-centered balls solve the weighted isoperimetric problem on cones \cite{CRS16}. We note that this concavity implies the log-concavity of $w$, that is, $w$ satisfies \eqref{star} with $K_w \geq 0$. The author in \cite{Lam21} subsequently employed optimal transport to sharpen weighted isoperimetric and Caffarelli--Kohn--Nirenberg inequalities under these assumptions. Euclidean logarithmic Sobolev inequalities with homogeneous weights have also been studied in \cite{BDK24}.

\medskip

Following the Poincar\'e inequality, it is natural to consider the logarithmic Sobolev inequality (LSI) on the weighted Gaussian measure $\mu_w$. Recall that the LSI for a measure $\mu$ takes the form
\[
\text{Ent}_{d\mu}(f^2) \leq 2C \int_\Sigma |\nabla f|^2 \, d\mu + D \int_\Sigma f^2 \, d\mu,
\]
where $\text{Ent}_{d\mu}(f) = \int_\Sigma f \log f \, d\mu - \left( \int_\Sigma f \, d\mu \right) \log \left( \int_\Sigma f \, d\mu \right)$, with constants $C > 0$ and $D \geq 0$. The inequality is called \emph{tight} when $D = 0$.

The classical logarithmic Sobolev inequality on $\mathbb{R}^n$ with Gaussian measure was established by Leonard Gross in his seminal paper~\cite{GRO75}, where he derived its tight form and used it to prove hypercontractivity of the Ornstein--Uhlenbeck semigroup~\cite{GRO75b}; this was built on earlier work by Stam~\cite{Stam59} and Federbush~\cite{Fede69}. This inequality strengthens the Poincar\'e inequality, as a tight LSI implies the latter with optimal constant. 

Our third principal aim in this article is to prove that the weighted Gaussian measure $\mu_w$ satisfies a tight logarithmic Sobolev inequality with optimal constant $C = \frac{1}{1+K_w}$. More precisely, we will prove that

\begin{theorem}\label{T4} Assume that $w$ satisfies \eqref{star}. For $f\in W^{1,q}_A(\Sigma,\mu_w)$, $\int_\Sigma |f|^q\,d\mu_w=1$,
\begin{equation}
\frac{2}{q}\int_\Sigma |f|^q\log |f|\,d\mu_w 
\le \frac{1}{1+K_w} \Bigl( \int_\Sigma |\nabla f|^q\,d\mu_w \Bigr)^{2/q}.
\end{equation}
\end{theorem}

In particular, we have

\begin{theorem}\label{T5} Assume that $w$ satisfies \eqref{star}. For $f\in W^{1,2}_A(\Sigma,\mu_w)$, we have 
\begin{equation}
    \text{Ent}_{d\mu_w}(f^2) \leq \frac{2}{1+K_w}\int_\Sigma |\nabla f|^2 \, d\mu_w. \label{LSIHG}
\end{equation}
Moreover, if $w$ is partial, equality is achieved by nontrivial functions. 
\end{theorem}

It is well known that the Gaussian logarithmic Sobolev inequality admits several equivalent formulations with respect to Lebesgue measure. In particular, via a suitable change of variables, the standard Gaussian LSI yields the following Euclidean logarithmic Sobolev inequality:
\begin{equation}\label{ElogS}
\int_{\mathbb{R}^n} |f|^2 \ln |f| \, dx \leq \frac{n}{4} \ln \left( \frac{2}{\pi n e} \int_{\mathbb{R}^n} |\nabla f|^2 \, dx \right),
\end{equation}
under the normalization $\int_{\mathbb{R}^n} |f|^2 \, dx = 1$. We refer, for instance, to~\cite{Wei78} for further details and related formulations.

Since their introduction, logarithmic Sobolev inequalities--both in Gaussian and Lebesgue settings--have played a central role in analysis and probability theory. They have been extensively studied, with applications spanning functional inequalities, concentration of measure phenomena, spectral theory for diffusions, partial differential equations, and mathematical physics. See, e.g.,~\cite{BGL14, Gentil03, Gross06, Gross25, Ledoux01, Wang05, FW-99, FW-97, FW-209} and the references therein (this list is far from exhaustive).

A natural question arises: Can we establish the equivalence between the logarithmic Sobolev inequality (LSI) with respect to the weighted Gaussian measure and the LSI with respect to the general weighted measure?

We answer this affirmatively in key cases. Let $d\nu = w(x) \, dx$ denote the log concave, homogeneous measure. In this case, one must have $K_w = 0$, which allows us to identify optimizers and obtain sharper insight. Specifically, we show that

\begin{theorem}\label{T6}
Assume that $w$ is log concave and homogeneous of degree $\alpha \geq 0$. Then we have the following LSI for the homogeneous weight: for all $f \in W^{1,2}_A(\Sigma,\mu_w)$,
\begin{equation}\label{LSIH}
\text{Ent}_{d\nu}(f^2) \leq \frac{n+\alpha}{2} \int_\Sigma f^2 \, d\nu \ \log \left( C_{\mathrm{LSIH}} \frac{\int_\Sigma |\nabla f|^2 \, d\nu}{\int_\Sigma f^2 \, d\nu} \right),
\end{equation}
where the sharp constant is given by $$C_{\mathrm{LSIH}} =\frac{4 C_w^{\frac{2}{n+\alpha}}}{e(n+\alpha)}. $$ Here $C_w = \left( \int_\Sigma w(x) e^{-\frac{|x|^2}{2}} \, dx \right)^{-1}$. 
The equality holds for Gaussian functions $f = Ae^{-\frac{|x|^2}{4}}$ where $A \in \mr$.

\end{theorem}

Moreover, equivalence holds fully in this case:

\begin{corollary}\label{C2}
Assume that $w$ is log concave and homogeneous of degree $\alpha \geq 0$. Then the homogeneous LSI~\eqref{LSIH} is equivalent to the homogeneous Gaussian LSI~\eqref{LSIHG}.
\end{corollary}

The paper is organized as follows: In Section 2, we apply the semigroup approach to establish generalized Beckner inequalities for weighted Gaussian measures. In Section 3, we use the Bakry-\'Emery $\Gamma$-Calculus approach and the duality method to investigate the gradient and $L^2$ stability of the Poincar\'e inequalities with weighted Gaussian measures, and use these results to study the stability of the Heisenberg Uncertainty Principle with homogeneous weights. In Section 4, we derive the logarithmic Sobolev inequalities with weighted Gaussian measures and use them to investigate the Euclidean logarithmic Sobolev inequalities with homogeneous log-concave weights.

\section{Beckner inequalities with weighted Gaussian measures-Proofs of Theorem \ref{T1} and Theorem \ref{T2}}    
We will now show that Beckner inequality holds for weighted Gaussian measure. 

\begin{proof}[Proof of Theorem \ref{T1}]
Consider the diffusion operator 
\[
\mathbf{L}_{w}:=\Delta-x\cdot\nabla+\nabla \log w\cdot\nabla
\]
associated to the homogeneous Gaussian measure $d\mu_w$ and the standard \emph{carré du champ} operator $\Gamma_w (u,v) = \nabla u \cdot \nabla v$. Note that $-\mathbf{L}_{w}$ is a non negative symmetric operator with respect to the measure $\mu_w$. The Dirichlet form is defined as follows: 
\begin{equation*}
    \mathcal{E}_w(u,v) := \int_\Sigma \nabla u \cdot \nabla v d\mu_w.
\end{equation*}
We also define the iterated \emph{carré du champ} operator $\Gamma_2$ by 
\begin{align*}
        \Gamma_2(f) &= \Gamma_2(f,f)= \frac{1}{2}\mathbf{L}_{w}\Gamma_w(f,f) - \Gamma_w(f,\mathbf{L}_{w}f) \\ 
        &=  \sum_{i,j}\left|f_{ij}\right|^2 + \sum_{i,j}f_j f_{jii} - \sum_{i,j}x_j f_i f_{ij} +\sum_{i,j} \frac{w_i}{w}f_{ij}f_j  \\ 
        &- \sum_{i,j}f_j f_{jii} + \sum_{i}\left|f_i\right|^2 + \sum_{i,j}x_j f_i f_{ij} - \sum_{i,j} \frac{w_{ij}w - w_i w_j}{w^2} f_if_j - \sum_{i,j}\frac{w_i}{w}f_{ij}f_j  \\
        &=  \sum_{i,j}\left|f_{ij}\right|^2 +  \sum_{i}\left|f_i\right|^2 - \sum_{i,j}\frac{w_{ij}w - w_i w_j}{w^2} f_if_j \\
        &= ||\nabla^2 f||_{F}^2 + |\nabla f|^2 - \nabla^2 \log w (\nabla f, \nabla f),
    \end{align*}
where $$||A||_F = \sqrt{\sum_{i,j}|a_{ij}|^2}$$ is the Frobenius norm of the matrix $A$. \\
Since $$-\nabla^2 \log w\geq K_w \ \mathbb{I},$$ we have 
\begin{align*}
    \Gamma_2(f) &\geq ||\nabla^2 f||_F^2 + |\nabla f|^2 +K_w|\nabla f|^2 \\
    & \geq ||\nabla^2 f||_F^2 +  (1+K_w) \ \Gamma(f). 
\end{align*}
So, the probability measure $d\mu_w$ satisfies Curvature-Dimension condition $CD(1+K_w, \infty)$.

Now, WLOG, we can assume that $f \geq 0$. Recall that $P_tf$ is the Markov semigroup with infinitesimal generator $\mathbf{L}_w$ and invariant reversible measure $\mu_w$.  Then consider the function

    $$\phi(t) = \left(\int_{\Sigma} (P_tf^p)^{\frac{q}{p}}d\mu_w\right)^{\frac{2}{q}}.$$
    Then $$\phi(0) = \left(\int_{\Sigma} (f^p)^{\frac{q}{p}}d\mu_w\right)^{\frac{2}{q}} =\left(\int_{\Sigma} f^q d\mu_w\right)^{\frac{2}{q}} =  ||f||_q^2 . $$
    And since the semigroup is ergodic we have $P_tf \to \int_\Sigma f d\mu$. Then 
    \begin{align*}
        \lim_{t\to \infty} \phi(t) = \lim_{t\to \infty}\left(\int_{\Sigma} (P_tf^p)^{\frac{q}{p}}d\mu_w\right)^{\frac{2}{q}} = \left(\int_{\Sigma} \left(\int_{\Sigma} f^p d\mu_w \right)^{\frac{q}{p}}d\mu_w\right)^{\frac{2}{q}}
        = \left(||f||_p^q\right)^{\frac{2}{q}} = ||f||_p^2 . 
    \end{align*}
    Thus we can write
    \begin{align*}
        \int_0^\infty -\phi'(t) dt = \phi(0) - \phi(\infty)  = ||f||_q^2 - ||f||_p^2 .
    \end{align*}
    Now we would like to bound the derivative. 
    \begin{align*}
        \phi'(t) &= \frac{d}{dt}\left(\int_{\Sigma} (P_tf^p)^{\frac{q}{p}}d\mu_w\right)^{\frac{2}{q}} \\& = \frac{2}{q}\underbrace{\left(\int_{\Sigma} (P_tf^p)^{\frac{q}{p}}d\mu_w\right)^{\frac{2}{q}-1}}_{\alpha(t)} \frac{d}{dt} \int_{\Sigma} (P_tf^p)^{\frac{q}{p}}d\mu_w \\
        &= \frac{2}{q}\alpha(t) \int_{\Sigma} \frac{q}{p}(P_tf^p)^{\frac{q}{p}-1} \frac{d}{dt}P_tf^p d\mu_w \\
        &= \frac{2}{p}\alpha(t) \int_{\Sigma} (P_tf^p)^{\frac{q}{p}-1} L_wP_tf^p d\mu_w .
    \end{align*}

Using properties of the diffusion operator $L_w$ we can rewrite 
    \begin{align*}
        \int_{\Sigma} (P_tf^p)^{\frac{q}{p}-1} L_w P_tf^p d\mu_w  &= -\int_\Sigma \nabla \left((P_tf^p)^{\frac{q}{p}-1}\right) \cdot \nabla \left(P_tf^p\right)  d\mu_w \\
        &= -\int_\Sigma \left(\frac{q}{p}-1\right)(P_tf^p)^{\frac{q}{p}-2} \nabla\left(P_tf^p\right) \cdot \nabla\left(P_tf^p\right) d\mu_w \\
        &= - \int_\Sigma \left(\frac{q}{p}-1\right)(P_tf^p)^{\frac{q}{p}-2} |\nabla\left(P_tf^p\right)|^2 d\mu_w .
    \end{align*}
Thus 
    $$- \phi'(t) = \frac{2}{p}\left(\frac{q}{p}-1\right)\alpha(t) \int_\Sigma (P_tf^p)^{\frac{q}{p}-2} |\nabla\left(P_tf^p\right)|^2 d\mu_w .$$
Under curvature dimension condition $CD(1+K_w,\infty)$ we have the strong gradient bounds \cite[Theorem 3.2.4]{BGL14} $$|\nabla\left(P_tf^p\right)|^2  \leq e^{-2(1+K_w)t} P_t(|\nabla (f^p)|)^2 = e^{-2(1+K_w)t} P_t(pf^{p-1} |\nabla f|)^2 .$$
So 
    \begin{align*}
        -\phi'(t) &\leq \frac{2}{p}\left(\frac{q}{p}-1\right)\alpha(t) \int_\Sigma (P_tf^p)^{\frac{q}{p}-2} e^{-2(1+K_w)t} P_t(pf^{p-1} |\nabla f|)^2 d\mu_w  \\
        &= 2e^{-2(1+K_w)t} (q-p)\alpha(t) \int_\Sigma (P_tf^p)^{\frac{q}{p}-2} P_t(f^{p-1} |\nabla f|)^2 d\mu_w .
    \end{align*} 

We can use kernel representation of the semigroup $P_t$ and Holder's inequality to get the bound $$P_t(f^{p-1} |\nabla f|)^2 \leq (P_tf^p)^{2 - \frac{2}{p}} P_t(|\nabla f|^p)^{\frac{2}{p}} .$$
Finally applying Holder's inequality again with $a = \frac{q}{q-2}$ and $b = \frac{q}{2}$ we get 
\begin{align*}
        -\phi'(t) &\leq 2e^{-2(1+K_w)t} (q-p)\alpha(t) \int_\Sigma (P_tf^p)^{\frac{q}{p}-\frac{2}{p}} P_t(|\nabla f|^p)^{\frac{2}{p}} d\mu_w \\
        &\leq 2e^{-2(1+K_w)t} (q-p)\left(\int_\Sigma (P_t|\nabla f|^p)^{ \frac{q}{p}} d\mu_w\right)^{\frac{2}{q}} .
    \end{align*} 
Since $ 1\leq p \leq q$, $\frac{q}{p} \geq 1$ so the function $g(s) = s^{\frac{q}{p}}$ is convex. Thus by Jensen's inequality and invariance we get 
$$-\phi'(t) \leq  2e^{-2(1+K_w)t} (q-p)\left(\int_\Sigma |\nabla f|^q d\mu_w\right)^{\frac{2}{q}} .$$
Integrating from $0$ to $\infty$ we get the desired result:
    \begin{align*}
        \int_0^\infty -\phi'(t) dt &\leq 2 (q-p) ||\nabla f||_q^2  \int_0^\infty e^{-2(1+K_w)t} dt \\
        &= 2 (q-p) ||\nabla f||_q^2 \frac{1}{2(1+K_w)}\\
        &= \frac{q-p}{1+K_w}||\nabla f||_q^2 .
    \end{align*}
\end{proof}

\begin{proof}[Proof of Theorem \ref{T2}] 
Since the measure $\mu_w$ satisfies Curvature Dimension condition $CD(1+K,\infty) $, by \cite[Proposition 7.6.1]{BGL14} we obtain Beckner inequality with $C = \frac{1}{1+K_w}$: 
    $$\frac{2}{2-q} \left(\int_\Sigma f^2 d\mu_w - \left(\int_\Sigma |f|^{q}d\mu_w \right)^{\frac{2}{q}}\right) \leq \frac{1}{1+K_w} \int_\Sigma |\nabla f|^2 d\mu_w .$$
    Next, we would like to show the constant in Beckner's Inequality is sharp. 
Take $u\in C_c^\infty$ such that $\int_{\Sigma} ud\mu_w = 0$ and $\int_{\Sigma} |u|^{2} d\mu_w = 1$. Consider $f = 1+ \varepsilon u$ where $\varepsilon$ is small enough so that $f > 0$. Then 
$$\int_{\Sigma} f^2 d\mu_w = \int_{\Sigma} 1 + 2\varepsilon u + \varepsilon ^2 |u|^{2} d\mu_w = 1 + \varepsilon^2 .$$
Consider $\varphi(t) = t^q$, then 
\begin{align*}
    \varphi(t) &= \varphi(a) + \varphi'(a)(t-a) + \frac{1}{2}\varphi''(a)(t-a)^2 + O(\varepsilon^3) \\
    \varphi(1+\varepsilon u)&= \varphi(1) + \varphi'(1)(1+\varepsilon u -1) + \frac{1}{2}\varphi''(1)(1+\varepsilon u -1)^2 + O(\varepsilon^3) \\
    &= 1 + q(\varepsilon u) +\frac{q(q-1)}{2} (\varepsilon u)^2 + O(\varepsilon^3) . 
\end{align*}
Integrating, we have 
\begin{align*}
    \int_{\Sigma} (1+\varepsilon u)^q d\mu_w &= \int_{\Sigma} 1 + q(\varepsilon u) +\frac{q(q-1)}{2} (\varepsilon u)^2 + O(\varepsilon^3) d\mu_w \\&= 1  + \frac{q(q-1)}{2}\varepsilon^2 + O(\varepsilon^3). 
\end{align*}
Doing the same for $\psi(t) = t^{\frac{2}{q}}$, we have $\psi' = \frac{2}{q}t^{\frac{2}{q}-1}$ and $\psi'' = \frac{2}{q}\left(\frac{2}{q}-1\right)t^{\frac{2}{q}-2}$. Thus
$$\psi(t) =  \psi(a) + \psi'(a)(t-a) + \frac{1}{2}\psi''(a)(t-a)^2 + O(\varepsilon^3).$$
Set $b = \frac{q(q-1)}{2}$ and $z = b\varepsilon^2 + O(\varepsilon^3)$. 
\begin{align*}
    \psi\left( 1  + z\right) &= 1 + \frac{2}{q}z + O(z^2) \\
    &= 1+ \frac{2}{q} \frac{q(q-1)}{2} \varepsilon^2 + \underbrace{O(\varepsilon^3) + O(z^2) }_{O(\varepsilon^3)}\\
    &= 1+ (q-1) \varepsilon^2 + O(\varepsilon^3) .
\end{align*}
Thus we can expand: 
\begin{align*}
    \int_\Sigma f^2  d\mu_{w} - \left(\int_\Sigma |f|^{q} d\mu_{w} \right)^{\frac{2}{q}} &= 1 + \varepsilon^2 - (1+ (q-1) \varepsilon^2 + O(\varepsilon^3) ) \\
&= (2-q)\varepsilon^2 + O(\varepsilon^3) .
\end{align*}
On the RHS we have 
$$\nabla f = \nabla (1+\varepsilon u) = \varepsilon\nabla u .$$
So 
$$\int_{\Sigma} |\nabla f|^2 d\mu_w = \varepsilon^2 \int_{\Sigma} |\nabla u|^2 d\mu_w.$$
Thus we have 
\begin{equation*}
    \varepsilon^2 + O(\varepsilon^3)  \leq \frac{1}{1+K_w} \varepsilon^2 \int_{\Sigma} |\nabla u|^2 d\mu_w.
\end{equation*}
Dividing by $\varepsilon^2$ and taking $\varepsilon \to 0$ we get:
$$\int_{\Sigma} |u|^{2} d\mu_w - \left(\int_{\Sigma} u d\mu_w\right)^2 =1\leq \frac{1}{1+K_w}\int_{\Sigma} |\nabla u|^2 d\mu_w .$$
If $w$ is partial we know that $K_w = 0$ and we recover the sharp Poincar\'e inequality. Thus the Beckner inequality must also be sharp. By taking $u_0$ to be the optimizer of the Poincar\'e inequality we get a sequence $f_n = 1 + \frac{1}{n}u_0$ that approaches equality. 
\end{proof}

\section{Poincar\'e Inequalities with weighted Gaussian measure and the Stability of HUP with homogeneous weight}
\subsection{Proofs of Theorems \ref{T2.1} and \ref{T3}}
In this subsection, we will use Theorem \ref{T1} and Bakry-\'Emery $\Gamma$-Calculus to establish Poincar\'e inequality with homogeneous Gaussian measure $d\mu_w$. Recall that \begin{equation*}
    d\mu_w = \frac{w(x) e^{-\frac{|x|^2}{2}}dx}{\int_\Sigma w(x) e^{-\frac{|x|^2}{2}}dx}. 
\end{equation*}
\begin{proof}[Proof of Theorem \ref{T2.1}]
    Instead, what if we expand around the function $f = 1+ \varepsilon g$ where $g \in C_c^\infty$ with $\int_\Sigma g d\mu_w = 0$ and for $\varepsilon $ small enough so that $f\geq 0$. Then on the RHS we have 
    $$\left(\int_\Sigma |\nabla f|^q d\mu_w\right)^{\frac{2}{q}} = \left(\int_\Sigma |\nabla( 1+\varepsilon g)|^q d\mu_w\right)^{\frac{2}{q}} = \left(\int_\Sigma |\varepsilon\nabla g|^q d\mu_w\right)^{\frac{2}{q}} = \varepsilon^2 \left(\int_\Sigma |\nabla g|^q d\mu_w\right)^{\frac{2}{q}}.$$
    On the LHS we would like to expand the term $||f||^2_q - ||f||^2_p$.
    We can use the Binomial/Taylor approximation $(1+x)^{\alpha} \approx 1 + \alpha x + \frac{1}{2}\alpha(\alpha -1)x^2 $: 
    \begin{align*}
        |1+\varepsilon g|^{q} &= 1 + \varepsilon g q + \frac{q(q-1)}{2} \varepsilon^2g^2 + O(\varepsilon^3). \\
        \int_\Sigma  |1+\varepsilon g|^{q} d\mu_w &= \int_\Sigma 1 + \varepsilon g q + \frac{q(q-1)}{2} \varepsilon^2g^2 + O(\varepsilon^3) d\mu_w \\
        &= 1 + \frac{q(q-1)}{2}\varepsilon^2 \int_\Sigma g^2 d\mu_w + O(\varepsilon^3).
    \end{align*}
    \begin{align*}
        \left(\int_\Sigma |1+\varepsilon g|^{q} d\mu_{w} \right)^{\frac{2}{q}}  &= \left(1 + \underbrace{\frac{q(q-1)}{2}\varepsilon^2 \int_\Sigma g^2 d\mu_w + O(\varepsilon^3}_{A})\right)^{\frac{2}{q}} \\
        &\approx 1 + \frac{2}{q}A + O(A^2) \\
        &= 1 + \frac{2}{q}\frac{q(q-1)}{2}\varepsilon^2 \int_\Sigma g^2 d\mu_w + O(\varepsilon^3) + O(\varepsilon^4) \\
        &= 1+ (q-1) \varepsilon^2 \int_\Sigma g^2 d\mu_w + O(\varepsilon^3).
    \end{align*}
    Thus we have 
    \begin{align*}
        ||f||^2_q - ||f||^2_p &= \left(\int_\Sigma f^{q} d\mu_{w} \right)^{\frac{2}{q}} - \left(\int_\Sigma f^p  d\mu_{w}\right)^\frac{2}{p} \\
        &= 1+ (q-1) \varepsilon^2 \int_\Sigma g^2 d\mu_w + O(\varepsilon^3) - \left(1+ (p-1) \varepsilon^2 \int_\Sigma g^2 d\mu_w + O(\varepsilon^3)\right)\\
        &= (q-p) \varepsilon^2 \int_\Sigma g^2 d\mu_w + O(\varepsilon^3).
    \end{align*}
    So Beckner inequality becomes
    \begin{align*}
        \frac{1}{q-p}\left( \left(\int_\Sigma f^{q} d\mu_{w} \right)^{\frac{2}{q}} - \left(\int_\Sigma f^p  d\mu_{w}\right)^\frac{2}{p}\right)  &= \frac{1}{q-p}\left(  (q-p) \varepsilon^2 \int_\Sigma g^2 d\mu_w + O(\varepsilon^3)\right) \\
        &= \varepsilon^2 \int_\Sigma g^2 d\mu_w + O(\varepsilon^3) \\
        &\leq \frac{1}{1+K_w} \varepsilon^2 \left(\int_\Sigma |\nabla g|^q d\mu_w\right)^{\frac{2}{q}}.
    \end{align*}
    Dividing by $\varepsilon^2$ we get
    $$\int_\Sigma g^2 d\mu_w + O(\varepsilon) \leq \frac{1}{1+K_w}  \left(\int_\Sigma |\nabla g|^q d\mu_w\right)^{\frac{2}{q}}.$$
    Taking the limit as $\varepsilon \to 0$ we get 
    $$\int_\Sigma g^2 d\mu_w = \text{Var}_{d\mu_w}(g) \leq \frac{1}{1+K_w}||\nabla g||^2_q.$$
\end{proof}

\begin{proof}[Proof of Theorem \ref{T3}]
As in the proof of Theorem \ref{T1}, the probability measure $d\mu_w$ satisfies Curvature-Dimension condition $CD(1+K_w, \infty)$. Therefore, by \cite[Proposition 4.8.1]{BGL14}, we obtain the Poincar\'e
inequality for homogeneous Gaussian measure $d\mu_{w}$ with constant $1 + K_w$ for $K_w > -1$. In other words, 
\begin{equation*}
    \int_\Sigma \left\vert \nabla f\right\vert ^{2}d\mu_{w}
        \geq (1+K_w)\int_\Sigma \left\vert f-\int_\Sigma 
        fd\mu_{w}\right\vert ^{2}d\mu_{w}.
\end{equation*}
We can also prove this by the duality approach that will be useful in studying the improvement of the Poincar\'e inequality. More precisely, consider the diffusion operator 
\[
\mathbf{L}_{w}:=\Delta-x\cdot\nabla+\nabla \log w\cdot\nabla
\]
associated to the homogeneous Gaussian measure $d\mu_w$. Assume that $\int_\Sigma fd\mu_w=0$ and let $u$ be the solution of Poisson equation $-\mathbf{L}_{w}u = f$. Note that $u$ satisfies the Neumann boundary condition: $\nabla u\cdot
\mathbf{\eta}=0\text{ on }\partial\Sigma$, where $\mathbf{\eta}$ is the outer normal vector of $\Sigma$.
See \cite[Lemma 5.4]{Bonn22}, for instance. Recall that with the standard \textit{carré du champ} operator $\Gamma_w (u,v) = \nabla u \cdot \nabla v$, we have $$\Gamma_2(u) = ||\nabla^2 u||_F^2 + |\nabla u|^2 - \nabla^2 \log w (\nabla u, \nabla u) .$$
Integrating both sides we get 
\begin{align*}
\int_\Sigma |f|^2 d\mu_w &= \int_\Sigma |\mathbf{L}_{w}u|^2 d\mu_w = \int_\Sigma \Gamma_2(u) d\mu_w \\
&=\int_\Sigma |\nabla^2 u|^2 d\mu_w +  \int_\Sigma |\nabla u|^2 d\mu_w - \int_\Sigma \nabla^2 \log w (\nabla u,\nabla u) d\mu \\
&\geq \int_\Sigma |\nabla^2 u|^2 d\mu_w + (1+K_w) \int_\Sigma |\nabla u|^2 d\mu_w\\
&\geq (1+K_w) \int_\Sigma |\nabla u|^2 d\mu_w.
\end{align*}
Using the identity 
$$\int_\Sigma |c\nabla u - \nabla f|^2 d\mu_w  = c^2\int_\Sigma |\nabla u |^2d\mu_w + \int_\Sigma |\nabla f|^2d\mu_w - 2c\int_\Sigma |f|^2 d\mu_w, $$
we get 
\begin{align*}
    \int_\Sigma|c\nabla u - \nabla f|^2 d\mu_w &\leq\frac{c^2}{1+K_w} \int_\Sigma |f|^2 d\mu_w+ \int_\Sigma |\nabla f|^2 d\mu_w - 2c\int_\Sigma |f|^2 d\mu_w \\
    &= \left(\frac{c^2}{1+K_w} - 2c\right)\int_\Sigma |f|^2 d\mu_w  + \int_\Sigma |\nabla f|^2 d\mu_w.
\end{align*}
By choosing $c=(1+K_w)$, we deduce that
$$\int_\Sigma |\nabla f|^2 d\mu_w - (1+K_w)\int_\Sigma |f|^2 d\mu_w \geq \int_\Sigma|(1+K_w)\nabla u - \nabla f|^2 d\mu_w \geq 0.$$

Now, suppose that the weight $w(x) = w(x_1,\dots,x_N)$ is independent of variables $x_1, \dots, x_i$. Then $K_w $ is identically $0$, and the optimizers are given by $f(x) = a_1x_1 + \dots + a_ix_i + b$. Indeed, we have $$\nabla f = (a_1, \dots, a_i, 0, \dots, 0)$$ and $$|\nabla f|^2 = \sum_{k=1}^i a_k^2.$$
On the RHS $$|f|^2 = \sum_{k=1}^i a_k^2 x_k^2 + 2\sum_{k<j}a_ka_jx_kx_j + 2b \sum_{k=1}^i a_k x_k
+ b^2.$$
\begin{align*}
    \left(\int_\Sigma f d\mu_w\right)^2 = \left(\int_\Sigma a_kx_k + b d\mu_w \right)^2 = b^2.
\end{align*}
Since $$\int_\Sigma x_k d\mu_w = \int_\Sigma x_k e^{-\frac{1}{2}x_k^2}dx_k = 0.$$
Thus \begin{align*}
    \int_\Sigma \left|f-\int_\Sigma f d\mu_w\right|^2 d\mu_w  &= \int_\Sigma \sum_{k=1}^i a_k^2 x_k^2 d\mu_w+ 2\sum_{k<j} \int_\Sigma a_ka_jx_kx_j d\mu_w\\&+2b \sum_{k=1}^i \int_\Sigma a_k x_kd\mu_w
    +\int_\Sigma b^2 d\mu_w - b^2 \\
    &= \sum_{k=1}^i a_k^2.
\end{align*}

Next, we refine the above analysis to establish an improved Poincar\'e inequality that captures gradient stability effects. This refinement reveals higher-order corrections tied to the weight's curvature $K_w$, yielding sharper variance control for functions in $W^{1,2}_A(\Sigma,\mu_w)$.

Assume that $\int_\Sigma fd\mu=0$. Let $u$ be the solution of Poisson equation $-\mathbf{L}_{w}u = f$. Since $$\Gamma_2(u) = ||\nabla^2 u||_F^2 + |\nabla u|^2 - \nabla^2 \log w (\nabla u, \nabla u),$$
by integrating both sides, we get 
\begin{align*}
        \int_\Sigma ||\nabla^2 u||_F^2 d\mu_w +  \int_\Sigma |\nabla u|^2 d\mu_w - \int_\Sigma \nabla^2 \log w (\nabla u,\nabla u) d\mu_w = \int_\Sigma \Gamma_2(u) d\mu_w
        &= \int_\Sigma |\mathbf{L}_{w}u|^2 d\mu_w \\
        &= \int_\Sigma |f|^2 d\mu_w.
    \end{align*}
Under our additional assumption  $$- \nabla^2 \log w (\nabla u,\nabla u) \geq K_w|\nabla u|^2 \geq 0,$$
and thus 
$$\int_\Sigma |f|^2 d\mu_w \geq  \int_\Sigma ||\nabla^2 u||_F^2 + (1+K_w)\int_\Sigma |\nabla u|^2 d\mu_w. $$
Now, we note that
\[
\int_\Sigma ||\nabla^{2}u||_{F}^{2}d\mu_{w}=\sum_{i}%
\int_\Sigma |\nabla\partial_{i}u|^{2}d\mu_{w}.
\]
By applying the Poincar\'{e} inequality with homogeneous Gaussian weight (Theorem
\ref{T1}), we get
\begin{align*}
\sum_{i}\int_\Sigma |\nabla\partial_{i}u|^{2}d\mu_{w}  &
\geq\sum_{i}(1+K_w)\int_\Sigma \left\vert \partial_{i}u-\int
 \partial_{i}ud\mu_{w}\right\vert ^{2}d\mu_{w}\\
&  =(1+K_w)\int_\Sigma |\nabla u-\int_\Sigma \nabla
ud\mu_{w}|^{2}d\mu_{w} .
\end{align*}
Also,
\begin{align*}
\int_\Sigma \partial_{i}ud\mu_{w}  &  =-\int_\Sigma u\mathbf{L}_{w}x_{i}d\mu_{w}\\
&  =-\int_\Sigma x_{i}\mathbf{L}_{w}ud\mu_{w}\\
&  =\int_\Sigma x_{i}f  d\mu_{w}.
\end{align*}
Therefore,
\[
\int_\Sigma ||\nabla^{2}u||_{F}^{2}d\mu_{w}\geq (1+K_w)\int_\Sigma
 \left\vert \nabla u-\int_\Sigma f xd\mu_{w}\right\vert
^{2}d\mu_{w}.%
\]
Combining this with our estimate we get
$$\int_\Sigma |f|^2 d\mu_w \geq (1+K_w)\int_\Sigma
 \left\vert \nabla u-\int_\Sigma f xd\mu_{w}\right\vert
^{2}d\mu_{w} + (1+K_w)\int_\Sigma |\nabla u|^2 d\mu_w.$$
Using the identity 
$$\int_\Sigma |c\nabla u - \nabla f|^2 d\mu_w  = c^2\int_\Sigma |\nabla u |^2d\mu_w + \int_\Sigma |\nabla f|^2d\mu_w - 2c\int_\Sigma |f|^2 d\mu_w $$
together with the estimate we rewrite
    \begin{align*}
        \int_\Sigma |c\nabla u - \nabla f|^2 d\mu_w & \leq \frac{c^2}{1+K_w} \left(\int_\Sigma |f|^2 d\mu_w  - (1+K_w)\int_\Sigma
        \left\vert \nabla u-\int_\Sigma f xd\mu_{w}\right\vert
        ^{2}d\mu_{w}\right) \\
        &+ \int_\Sigma |\nabla f|^2d\mu_w - 2c\int_\Sigma |f|^2 d\mu_w \\
        &= \left(\frac{c^2}{1+K_w}-2c\right)\int_\Sigma |f|^2 d\mu_w  -c^2\int_\Sigma
        \left\vert \nabla u-\int_\Sigma f xd\mu_{w}\right\vert
        ^{2}d\mu_{w}+ \int_\Sigma |\nabla f|^2d\mu_w,
    \end{align*}
    and thus 
    \begin{align*}
        \int_\Sigma |\nabla f|^2d\mu_w - \left(2c - \frac{c^2}{1+K_w} \right)\int_\Sigma |f|^2 d\mu_w &\geq \int_\Sigma |c\nabla u - \nabla f|^2 d\mu_w +  c^2\int_\Sigma
        \left\vert \nabla u-\int_\Sigma f xd\mu_{w}\right\vert
        ^{2}d\mu_{w} \\
        &\geq \frac{1}{2}\int_\Sigma \left|\nabla f - c \int_\Sigma fx d\mu_w\right|^2 d\mu_w \\
        &\geq \frac{1}{2} \int_\Sigma \left|\nabla \left[f - c \int_\Sigma fx d\mu_w\cdot x\right]\right|^2 d\mu_w.
    \end{align*}
By choosing $c = (1+K_w)$, we deduce that 
$$\int_\Sigma |\nabla f|^2d\mu_w - (1+K_w)\int_\Sigma |f|^2 d\mu_w \geq \frac{1}{2} \int_\Sigma \left|\nabla \left[f - (1+K_w)\int_\Sigma fx d\mu_w\cdot x\right]\right|^2 d\mu_w.$$
\end{proof}

\subsection{Scale-dependent versions of the Poincar\'e inequality}

In this subsection, we assume that $w$ is homogeneous of degree $\alpha \geq 0$. As shown in \cite{CFLL24}, Theorems \ref{T3} is not enough to study the stability of the HUP. We need a version of scale-dependent Poincar\'e inequality with homogeneous Gaussian weight. In other words, for $\lambda > 0$ we can define a scale dependent homogeneous Gaussian measure $d\mu_{w,\lambda} =  \frac{w(x)e^{-\frac{|x|^2}{2\lambda^2}}dx}{\int_\Sigma  w(x)e^{-\frac{|x|^2}{2\lambda^2}} dx}$. Let $X_{w,\lambda}$ be the completion of $A$ under the norm $\left(\int_\Sigma | f|^2 d\mu_{w,\lambda}\right)^{\frac{1}{2}} + \left(\int_\Sigma  |\nabla f|^2 d\mu_{w,\lambda}\right)^{\frac{1}{2}}$. Then we establish the following Poincar\'e inequality: 
\begin{theorem}
    \label{T3.1}Assume that $w$ satisfies \eqref{star}. For $f\in X_{\lambda,w}$, we have
    \begin{align}
        \int_\Sigma \left|\nabla f \right|^{2} d\mu_{w,\lambda} &\geq
        \frac{1+K_w}{\lambda^2}\int_\Sigma \left\vert f-\int_\Sigma 
        f d\mu_{w,\lambda} \right\vert ^{2} d\mu_{w,\lambda} \nonumber \\
        &\geq \frac{1+K_w}{\lambda^2}\inf_{c} \int_\Sigma |f-c|^2  d\mu_{w,\lambda}. \label{SPIH}%
    \end{align}
    Moreover if $w$ is partial, then \eqref{SPIH} is achieved by non-constant functions.
\end{theorem}

\begin{proof}[Proof of Theorem \ref{T3.1}]
    We have from Theorem \ref{T3} that
    \begin{equation*}
        \int_\Sigma |\nabla f|^{2} w(x)e^{-\frac{1}{2}|x|^{2}}
        dx\geq 
         (1+K_w)\int_\Sigma \left\vert f-\frac{1}{\int_\Sigma  w(x)e^{-\frac{1}{2}|x|^{2}}dx}\int_\Sigma 
        f w(x)e^{-\frac{1}{2}|x|^{2}}dx\right\vert ^{2} w(x)e^{-\frac{1}{2}|x|^{2}
        }dx.
    \end{equation*}
    Let $f(x) = g\left(\frac{x}{\lambda}\right)$ then $\nabla f = \frac{1}{\lambda} g\left(\frac{x}{\lambda}\right)$. Then the expression becomes 
    
    \begin{gather*}
        \int_\Sigma \frac{1}{\lambda^2}\left|\nabla g\left(\frac{x}{\lambda}\right) \right|^{2} w(x)e^{-\frac{1}{2}|x|^{2}}
        dx \\
        \geq (1+K_w)
         \int_\Sigma \left\vert g\left(\frac{x}{\lambda}\right)-\frac{1}{\int_\Sigma  w(x)e^{-\frac{1}{2}|x|^{2}}dx}\int_\Sigma 
        g\left(\frac{x}{\lambda}\right) w(x)e^{-\frac{1}{2}|x|^{2}}dx\right\vert ^{2} w(x)e^{-\frac{1}{2}|x|^{2}
        }dx.
    \end{gather*}

    Next we let $y= \frac{x}{\lambda}$. Then $dx = \lambda^n dy$ and the expression becomes

    \begin{gather*}
        \int_\Sigma \frac{1}{\lambda^2}\left|\nabla g(y) \right|^{2} w(\lambda y)e^{-\frac{1}{2}|\lambda y|^{2}}
         \lambda^n dy \\
        \geq (1+K_w)
         \int_\Sigma \left\vert g(y)-\frac{1}{\int_\Sigma  w(\lambda y)e^{-\frac{1}{2}|\lambda y|^{2}}\lambda^n dy}\int_\Sigma 
        g\left(y\right) w(\lambda y)e^{-\frac{1}{2}|\lambda y|^{2}}dy\right\vert ^{2} w(\lambda y)e^{-\frac{1}{2}|\lambda y|^{2}
        }\lambda^n dy.
    \end{gather*}

    From homogeneity of $w$, we have $w(\lambda y) = \lambda^\alpha w(y)$. So we get

    \begin{gather*}
        \int_\Sigma \frac{\lambda^{n+\alpha}}{\lambda^2}\left|\nabla g(y) \right|^{2} 
        w(y)e^{-\frac{1}{2}|\lambda y|^{2}} \, dy \\
        \geq (1+K_w) \int_\Sigma \left| 
        g(y) - \frac{1}{\lambda^{n+\alpha}\int_\Sigma w(y)e^{-\frac{1}{2}|\lambda y|^{2}} \, dy}
        \int_\Sigma g(y) w(y)e^{-\frac{1}{2}|\lambda y|^{2}} \, dy 
        \right|^{2} 
        w(y)e^{-\frac{1}{2}|\lambda y|^{2}} \lambda^{n+\alpha} \, dy.
    \end{gather*}

    Finally we can divide both sides by $\lambda^{n+\alpha}$ and changing $\lambda =\frac{1}{\eta}$ we get

    \begin{gather*}
        \eta^2 \int_\Sigma \left|\nabla g \right|^{2} w(y)e^{-\frac{1}{2 \eta^2}|y|^{2}}
          dy \geq(1+K_w)
         \int_\Sigma \left\vert g-\frac{1}{\int_\Sigma  w(y)e^{-\frac{1}{2\eta^2}|y|^{2}} dy}\int_\Sigma 
        g w(y)e^{-\frac{1}{2\eta^2}|y|^{2}}dy\right\vert ^{2} w(y)e^{-\frac{1}{2\eta^2}|y|^{2}
        }dy.
    \end{gather*}
    So letting $d\mu_{w,\eta} = \frac{1}{\int_\Sigma  w(y)e^{-\frac{1}{2\eta^2}|y|^{2}} dy} w(y)e^{-\frac{1}{2\eta^2}|y|^2}dy$, we can rewrite this as 
    \begin{gather*}
        \eta^2 \int_\Sigma \left|\nabla g \right|^{2} d\mu_{w,\eta} \geq(1+K_w)
         \int_\Sigma \left\vert g-\int_\Sigma 
        g d\mu_{w,\eta} \right\vert ^{2} d\mu_{w,\eta}.
    \end{gather*}
    Thus we get 
    \begin{equation*}
        \int_\Sigma \left|\nabla g \right|^{2} d\mu_{w,\eta} \geq \frac{1+K_w}{\eta^2}\inf_{c} \int_\Sigma |g-c|^2  d\mu_{w,\eta}. 
    \end{equation*}
\end{proof}

Similarly we can can obtain improved scale-dependent Poincar\'e inequality with homogeneous Gaussian weight: 

\begin{theorem}
    \label{T4.1}Assume that $w$ satisfies \eqref{star}. For $\lambda > 0$ and $f\in X_{\lambda,w}$, we have 
    \begin{align}
    &\lambda^2 \int_\Sigma |\nabla f|^2 d\mu_{w,\lambda} \nonumber\\  & \geq (1+K_w)\inf_{c}\int_\Sigma |f- c|^2 w(x)e^{-\frac{|x|^2}{2 \lambda^2}} dx + \frac{1+K_w}{2}\inf_{c, \mathbf{d}} \int_\Sigma \left|f - (c+\mathbf{d}\cdot x)\right|^2 w(x) e^{-\frac{|x|^2}{2 \lambda^2}}dx.
    \end{align}
\end{theorem}
 
\begin{proof}[Proof of Theorem \ref{T4.1}] Theorem \ref{T3} and Corollary \ref{C0} gives us
    \begin{align}
    &\int_\Sigma \left\vert \nabla f\right\vert ^{2}d\mu
        _{w}-(1+K_w)\int_\Sigma \left\vert f-\int_\Sigma fd\mu_{w}\right\vert ^{2}d\mu_{w} \nonumber\\
        &\geq \frac{1+K_w}{2} \int_\Sigma \bigg| f - \int_\Sigma f \ d\mu_w - (1+K_w)\int_\Sigma fx \ d\mu_w \cdot x + (1+K_w)\int_\Sigma f \ d\mu_w\int_\Sigma x \ d\mu_w \cdot x \nonumber \\
        &- (1+K_w)\int_\Sigma f \ d\mu_w \left|\int_\Sigma x \ d\mu_w\right|^2 +(1+K_w)\int_\Sigma fx \ d\mu_w \cdot \int_\Sigma x \ d\mu_w \bigg| ^2 d\mu_w \nonumber .
\end{align}
Let $f(x) = g(\lambda x) $, then as before 
\begin{gather*}
    \int_\Sigma |\nabla f|^2 d\mu_w - (1+K_w)\int_\Sigma \left|f - \int_\Sigma fd\mu_w\right|^2 d\mu_w \\
    = \lambda^{2-n-\alpha} \int_\Sigma |\nabla g|^2 w(x) e^{-\frac{|x|^2}{2\lambda^2}} dx - \lambda^{-n-\alpha}(1+K_w) \int_\Sigma \left|g - C(\lambda, g)\right|^2 w(x) e^{-\frac{|x|^2}{2\lambda^2}} dx,
\end{gather*}
and 
\begin{align*}
     &\frac{1+K_w}{2} \int_\Sigma \bigg| f - \int_\Sigma f \ d\mu_w -(1+K_w) \int_\Sigma fx \ d\mu_w \cdot x + (1+K_w)\int_\Sigma f \ d\mu_w\int_\Sigma x \ d\mu_w \cdot x \nonumber \\
        &- (1+K_w)\int_\Sigma f \ d\mu_w \left|\int_\Sigma x \ d\mu_w\right|^2 +(1+K_w)\int_\Sigma fx \ d\mu_w \cdot \int_\Sigma x \ d\mu_w \bigg| ^2 d\mu_w \\
        &= \frac{1+K_w}{2} \int_\Sigma \bigg|g - C(\lambda,g) - \mathbf{D}(\lambda,g) \cdot x \\
        &+ C(\lambda,g) 
 \int_\Sigma x d\mu_w \cdot x - C(\lambda,g)\left|\int_\Sigma x \ d\mu_w\right|^2 + \mathbf{D}(\lambda, g) \cdot \int_\Sigma x d\mu_w  \bigg|^2 w(x) e^{-\frac{|x|^2}{2\lambda^2}} dx
\end{align*}
for some $C(\lambda, g) \in \mr$ and $\mathbf{D}(\lambda,g) \in \mr^n$. Thus we have: 
\begin{align*}
    &\lambda^{2} \int_\Sigma |\nabla g|^2 w(x) e^{-\frac{|x|^2}{2\lambda^2}} dx \geq \inf_{C,\mathbf{D}}(1+K_w) \int_\Sigma \left|g - C\right|^2 w(x) e^{-\frac{|x|^2}{2\lambda^2}} dx  \\
    &+\frac{1+K_w}{2} \int_\Sigma \bigg|g - C - \mathbf{D} \cdot x \\
    & + C 
 \int_\Sigma x d\mu_w \cdot x - C\left|\int_\Sigma x \ d\mu_w\right|^2 + \mathbf{D} \cdot \int_\Sigma x d\mu_w  \bigg|^2 w(x) e^{-\frac{|x|^2}{2\lambda^2}} dx \\
 &\geq (1+K_w)\inf_{c}\int_\Sigma |g- c|^2 w(x) e^{-\frac{|x|^2}{2 \lambda^2}}dx + \frac{1+K_w}{2}\inf_{c, \mathbf{d}} \int_\Sigma \left|g - (c+\mathbf{d}\cdot x) \right|^2 w(x)e^{-\frac{|x|^2}{2\lambda^2}} dx.
\end{align*}
\end{proof}

\subsection{Stability of HUP with homogeneous weight}

Assume that $w$ is homogeneous of degree $\alpha \geq 0$. Now, we will check the Identity for HUP with homogeneous weight. 

\begin{proposition}
     \label{UNP}For $f\in S_w, \ f \neq 0$ and $\lambda =  \left(\frac{\int_\Sigma |f|^{2}|x|^2 w(x) dx}{\int_\Sigma |\nabla f|^2 w(x) dx}\right)^{\frac{1}{4}}$, we have that 
    \begin{align*}
         \left(\int_\Sigma |\nabla f|^2 w(x)dx\right)^{\frac{1}{2}}\left(\int_\Sigma |f|^{2}| x|^2 w(x)dx\right)^{\frac{1}{2}} &-\frac{n+\alpha}{2} \int_\Sigma |f|^{2}w(x)dx \\
        &= \frac{\lambda^2}{2} \int_\Sigma \left| \nabla\left(f e^{\frac{|x|^2}{2 \lambda^2}}\right)\right|^2 e^{-\frac{|x|^2}{\lambda^2}} w(x) dx.
    \end{align*}
\end{proposition}

Here $S_w$ is the completion of $A$ under the norm $\left(\int_\Sigma |\nabla f|^2 w(x) dx\right)^{\frac{1}{2}} + \left(\int_\Sigma |x|^2 |f|^{2} w(x)dx\right)^{\frac{1}{2}}$.
\begin{proof}[Proof of Proposition \ref{UNP}]
    Checking directly we have
    \begin{align*}
         \frac{\lambda^2}{2} \int_\Sigma \left| \nabla\left(f e^{\frac{|x|^2}{2 \lambda^2}}\right)\right|^2 e^{-\frac{|x|^2}{\lambda^2}} w(x) dx = \frac{\lambda^2}{2} \int_\Sigma |\nabla f|^2 w(x) dx \\
        + \frac{1}{2\lambda^2} \int_\Sigma |f|^{2}|x|^2 w(x) dx + \int_\Sigma f(\nabla f \cdot x) w(x) dx. \notag 
    \end{align*}
    For the last term we can use integration by parts to get
    \begin{align*}
        \int_\Sigma f(\nabla f \cdot x) w(x) dx &= -\frac{1}{2}\int_\Sigma |f|^{2} \text{div}(x w(x)) dx \\
        &= -\frac{1}{2} \int_\Sigma |f|^{2}\left(nw + x\cdot\nabla w\right)dx \\
        &= -\frac{n+\alpha}{2} \int_\Sigma |f|^{2}w(x)dx .
    \end{align*}
    Define 
    $$G(\lambda) = \frac{\lambda^2}{2} \int_\Sigma |\nabla f|^2 w(x) dx \\
    + \frac{1}{2\lambda^2} \int_\Sigma |f|^{2}|x|^2 w(x) dx  -\frac{n+\alpha}{2} \int_\Sigma |f|^{2}w(x)dx.$$
    $G$ is minimized at $\lambda = \left(\frac{\int_\Sigma |f|^{2}|x|^2 w(x) dx}{\int_\Sigma |\nabla f|^2 w(x) dx}\right)^{\frac{1}{4}}$. Using this value of $\lambda$ we get the identity: 
    \begin{align*}
         \left(\int_\Sigma |\nabla f|^2 w(x)dx\right)^{\frac{1}{2}}\left(\int_\Sigma |f|^{2}| x|^2 w(x)dx\right)^{\frac{1}{2}} &-\frac{n+\alpha}{2} \int_\Sigma |f|^{2}w(x)dx \\
        &= \frac{\lambda^2}{2} \int_\Sigma \left| \nabla\left(f e^{\frac{|x|^2}{2 \lambda^2}}\right)\right|^2 e^{-\frac{|x|^2}{\lambda^2}} w(x) dx.
    \end{align*}
\end{proof}

We are now ready to establish stability and improved stability of the Heisenberg Uncertainty Principle with homogeneous measure. More precisely, take the distance function
$$d_w (f, E_{HUPW}) = \inf_{c,\lambda\neq 0} \left(\int_\Sigma \left|f - ce^{-\frac{|x|^2}{2\lambda^2}}\right|^2 w(x)dx\right)^{\frac{1}{2}}$$
then we have the following: 
\begin{theorem}
    \label{S1.1}Assume that $w$ satisfies \eqref{star}. Let $f\in S_w$, then 
    \begin{equation}
        \delta_w(f) \geq (1+K_w) d^2_w \left(f, E_{HUPW}\right).
    \end{equation}
    Moreover, if $w$ is partial, the equality is achieved by non-trivial functions $f\notin E_{HUPW}$.
\end{theorem}
And for $$\Tilde{d}_w(f,F) = \inf_{c,\mathbf{d},\lambda \neq 0} \left(\int_\Sigma  \left|f - (c+\mathbf{d}\cdot x) e^{-\frac{|x|^2}{2\lambda^2}}\right|^2 w(x) dx \right)^{\frac{1}{2}}$$
we have improved stability: 
\begin{theorem}
    \label{S2}Assume that $w$ satisfies \eqref{star}. Let $\lambda > 0$ and $f\in S_w$ then we have: 
    $$\delta_w(f) -(1+K_w) d_w^2(f,E_{HUPW}) \geq \frac{1+K_w}{2} \Tilde{d}^2_w(f,F).$$
\end{theorem}
\begin{proof}[Proof of Theorem \ref{S1.1} (Theorem \ref{S1})] 
We have from Theorem \ref{SPIH}
\begin{align*}
        \int_\Sigma \left|\nabla f \right|^{2} d\mu_{w,\lambda} 
        \geq \frac{1+K_w}{\lambda^2}\inf_{c} \int_\Sigma |f-c|^2  d\mu_{w,\lambda},
\end{align*}
where 
$$d\mu_{w,\lambda} =  \frac{w(x)e^{-\frac{|x|^2}{2\lambda^2}}dx}{\int_\Sigma  w(x)e^{-\frac{|x|^2}{2\lambda^2}} dx}.$$
Changing $\lambda = \frac{\lambda}{\sqrt{2}}$ we get 
\begin{align*}
        \int_\Sigma \left|\nabla f \right|^{2} w(x)e^{-\frac{|x|^2}{\lambda^2}}dx \geq \frac{2(1+K_w)}{\lambda^2}\inf_{c} \int_\Sigma |f-c|^2  w(x)e^{-\frac{|x|^2}{\lambda^2}}dx.
\end{align*}
Applying this to the HUP we get 
\begin{align*}
         \left(\int_\Sigma |\nabla f|^2 w(x)dx\right)^{\frac{1}{2}}\left(\int_\Sigma |f|^{2}| x|^2 w(x)dx\right)^{\frac{1}{2}} &-\frac{N+\alpha}{2} \int_\Sigma |f|^{2}w(x)dx \\
        &= \frac{\lambda^2}{2} \int_\Sigma \left| \nabla\left(f e^{\frac{|x|^2}{2 \lambda^2}}\right)\right|^2 e^{-\frac{|x|^2}{\lambda^2}} w(x) dx \\
        &\geq (1+K_w) \inf_c \int_\Sigma \left|fe^{\frac{|x|^2}{2\lambda^2}}-c\right|^2  w(x)e^{-\frac{|x|^2}{\lambda^2}}dx \\
        &= (1+K_w)  \inf_{c} \int_\Sigma \left|f-ce^{-\frac{|x|^2}{\lambda^2}}\right|w(x)dx.
    \end{align*}
Now assume that $w$ is partial and WLOG is independent from $x_n$. Let $f=x_ne^{-\frac{|x|^2}{2}} \notin E_{HUPW}$. 
$$\int_\Sigma |x|^2|f|^2w(x) dx = \int_\Sigma |x|^2 x_n^2e^{-|x|^2} w(x)dx .$$
Let $w_2(x) = x_n^2 w(x)$. We see that $w_2(x)$ is homogeneous of degree $\alpha+2$. Set $V = x e^{-|x|^2}w_2(x)$, then 
\begin{align*}
    \text{div}(V) = n e^{-|x|^2}w_2(x) - 2|x|^2 e^{-|x|^2}w_2(x) + e^{-|x|^2} \underbrace{x\cdot \nabla w_2(x)}_{=(\alpha+2 )w_2(x)}.
\end{align*}
From Divergence theorem and decay at $\infty$, we have 
\begin{align*}
    \int_\Sigma \text{div}(V) dx = 0 &= \int_\Sigma n e^{-|x|^2}w_2(x) - 2|x|^2 e^{-|x|^2}w_2(x) + e^{-|x|^2} (\alpha+2 )w_2(x) dx \\
    &= (n+\alpha+2) \int_\Sigma e^{-|x|^2}w_2(x)dx - 2\int_\Sigma |x|^2e^{-|x|^2}w_2(x)dx.
\end{align*}
So
$$\int_\Sigma |x|^2e^{-|x|^2}w_2(x)dx = \frac{n+\alpha+2}{2} \int_\Sigma e^{-|x|^2}w_2(x)dx. $$
We also know that $\nabla f = (e_n - x_nx)e^{-\frac{|x|^2}{2}}$, so 
\begin{align*}
            \int_\Sigma |\nabla f|^2 w(x)dx &= \int_\Sigma |e_n-x_nx|^2w(x)e^{-|x|^2}dx \\
            &= \int_\Sigma (1-2x_n^2+x_n^2|x|^2)w(x)e^{-|x|^2}dx. 
\end{align*}
Let $Z = \left(0,\dots, 0, x_nw(x)e^{-|x|^2}\right)$, then 
$$\text{div}(Z) = w(x)e^{-|x|^2} - 2x_n^2 e^{-|x|^2}w(x) + x_n e^{-|x|^2} \p_nw(x).$$
Since $w(x) $ is independent of $x_n$, the last term is zero. Using divergence theorem again we get 
$$2\int_\Sigma x_n^2 e^{-|x|^2}w(x) dx =  \int_\Sigma  e^{-|x|^2}w(x)dx.$$
Then we see that 
\begin{align*}
            \int_\Sigma |\nabla f|^2 w(x)dx &= \int_\Sigma x_n^2|x|^2w(x)e^{-|x|^2}dx  = \frac{n+\alpha+2}{2} \int_\Sigma e^{-|x|^2}x_n^2 w(x)dx. 
\end{align*}
Thus on the LHS of HUP we have 

\begin{gather*}
    \left(\frac{n+\alpha+2}{2} \int_\Sigma e^{-|x|^2}x_n^2 w(x)dx \right)^{\frac{1}{2}}\left(\frac{n+\alpha+2}{2} \int_\Sigma e^{-|x|^2}x_n^2 w(x)dx \right)^{\frac{1}{2}} - \frac{n+\alpha}{2}\int_\Sigma e^{-|x|^2}x_n^2w(x)dx \\
    = \int_\Sigma e^{-|x|^2}x_n^2 w(x)dx.
\end{gather*}
On the RHS we have 
$$(1+K_w)\inf_{c} \int_\Sigma \left|f-ce^{-\frac{|x|^2}{\lambda^2}}\right|w(x)dx. $$
The assumption that $w$ is independent of $x_n$ forces $K_w=0$. We also have 
$$\lambda = \left(\frac{\int_\Sigma |x|^2|f|^2w(x)dx}{\int_\Sigma |\nabla f|^2 w(x)dx}\right)^{\frac{1}{4}} = 1.$$
So 
\begin{align*}
        RHS = \inf_c\left(\int_\Sigma \left(x_n^2e^{-|x|^2} - \underbrace{2cx_ne^{-|x|^2}}_{\text{odd in $x_n$}} + c^2e^{-|x|^2}\right)w(x)dx \right)
        &= \int_\Sigma x_n^2e^{-|x|^2}w(x)dx = LHS.
\end{align*}
\end{proof}
\begin{proof}[Proof of Theorem \ref{S2}]
We have from Proposition \ref{UNP} and Theorem \ref{T4.1}
    \begin{align*}
         &\left(\int_\Sigma |\nabla f|^2 w(x)dx\right)^{\frac{1}{2}}\left(\int_\Sigma |f|^{2}| x|^2 w(x)dx\right)^{\frac{1}{2}} -\frac{N+\alpha}{2} \int_\Sigma |f|^{2}w(x)dx \\
        &= \frac{\lambda^2}{2} \int_\Sigma \left| \nabla\left(f e^{\frac{|x|^2}{2 \lambda^2}}\right)\right|^2 e^{-\frac{|x|^2}{\lambda^2}} w(x) dx \\
        &\geq (1+K_w)\inf_{c, \lambda \neq 0}\int_\Sigma |f- ce^{-\frac{|x|^2}{2 \lambda^2}}|^2 w(x) dx + \frac{1+K_w}{2}\inf_{c, \mathbf{d}, \lambda \neq 0} \int_\Sigma \left|f - (c+\mathbf{d}\cdot x) e^{-\frac{|x|^2}{2\lambda^2}}\right|^2 w(x) dx .
    \end{align*}
\end{proof}

\section{Weighted logarithmic Sobolev inequalities}
\subsection{Proofs of Theorem \ref{T4}, Theorem \ref{T5}, Theorem \ref{T6} and Corollary \ref{C2}}
In this subsection, we will use Curvature Dimension condition to establish the logarithmic Sobolev inequality for weighted Gaussian measure.

\begin{proof}[Proof of Theorem \ref{T4}]
WLOG, assume that $f \geq 0$. We have from Theorem \ref{T1}
\begin{gather*}
        \frac{1}{q-p}\left( \left(\int_\Sigma f^{q} d\mu_{w} \right)^{\frac{2}{q}} - \left(\int_\Sigma f^p  d\mu_{w}\right)^\frac{2}{p}\right)  \leq \frac{1}{1+K_w} \left(\int_\Sigma |\nabla f|^q d\mu_w\right)^{\frac{2}{q}}.
    \end{gather*}
    Consider $$\phi(t) = \left(\int_\Sigma f^t  d\mu_{w}\right)^\frac{2}{t} \implies \log \phi = \frac{2}{t} \log \left(\int_\Sigma f^t d\mu_w\right).$$ 
    Then $$\lim_{p\to q} \frac{\phi(q) -\phi(p)}{q-p} = \phi'(q).$$
    \begin{align*}
        (\log \phi )' = \frac{\phi'}{\phi} &= -\frac{2}{t^2} \log \left(\int_\Sigma f^t d\mu_w\right) + \frac{2}{t}\frac{d}{dt}\log \left(\int_\Sigma f^t d\mu_w\right) \\
        &= -\frac{2}{t^2}\log \left(\int_\Sigma f^t d\mu_w\right) + \frac{2}{t}\frac{\int_\Sigma f^t \log f d\mu_w}{\int_\Sigma f^t d\mu_w }.
    \end{align*}
    So 
    \begin{align*}
        \phi'(t) &= \left(\int_\Sigma f^t  d\mu_{w}\right)^\frac{2}{t}   \left(-\frac{2}{t^2}\log \left(\int_\Sigma f^t d\mu_w\right)  + \frac{2}{t}\frac{\int_\Sigma f^t \log f d\mu_w}{\int_\Sigma f^t d\mu_w }\right) \\
        &= -\frac{2}{t^2}\left(\int_\Sigma f^t  d\mu_{w}\right)^{\frac{2}{t}}\log \left(\int_\Sigma f^t d\mu_w\right) + \frac{2}{t} \left(\int_\Sigma f^t \log f  d\mu_{w}\right) \left(\int_\Sigma f^t  d\mu_{w}\right)^{\frac{2}{t} -1}.
    \end{align*}
    So the inequality becomes 
    $$-\frac{2}{q^2}\left(\int_\Sigma f^q  d\mu_{w}\right)^{\frac{2}{q}}\log \left(\int_\Sigma f^q d\mu_w\right) + \frac{2}{q} \left(\int_\Sigma f^q \log f  d\mu_{w}\right) \left(\int_\Sigma f^q  d\mu_{w}\right)^{\frac{2}{q} -1} \leq \frac{1}{1+K_w}\left(\int_\Sigma |\nabla f|^q d\mu_w\right)^{\frac{2}{q}}.$$
    Factoring out $ \frac{2}{q^2}\left(\int_\Sigma f^q  d\mu_{w}\right)^{\frac{2}{q} -1} $, we get 
    \begin{align*}
        &\frac{2}{q^2}\left(\int_\Sigma f^q  d\mu_{w}\right)^{\frac{2}{q} -1} \left(q\int_\Sigma f^q \log f  d\mu_{w}- \int_{\Sigma} f^q d\mu_w \log \left(\int_\Sigma f^q d\mu_w\right)\right) \\
        =& \frac{2}{q^2}\left(\int_\Sigma f^q  d\mu_{w}\right)^{\frac{2}{q} -1} \left(\int_\Sigma f^q \log f^q  d\mu_{w}- \int_{\Sigma} f^q d\mu_w \log \left(\int_\Sigma f^q d\mu_w\right)\right) \\
        =& \frac{2}{q^2}\left(\int_\Sigma f^q  d\mu_{w}\right)^{\frac{2}{q} -1} \text{Ent}_{d\mu_w}(f^q).
    \end{align*}
    So the inequality becomes 
    $$\frac{2}{q^2}\left(\int_\Sigma f^q  d\mu_{w}\right)^{\frac{2}{q} -1} \text{Ent}_{d\mu_w}(f^q) \leq \frac{1}{1+K_w}\left(\int_\Sigma |\nabla f|^q d\mu_w\right)^{\frac{2}{q}}.$$
    If we normalize $f$ so that $\int_\Sigma f^q d\mu_w = 1$, then the inequality simplifies to 
    $$\frac{2}{q}\int_\Sigma f^q \log f  d\mu_{w} \leq \frac{1}{1+K_w}\left(\int_\Sigma |\nabla f|^q d\mu_w\right)^{\frac{2}{q}}.$$ 
\end{proof}

\begin{proof}[Proof of Theorem \ref{T5}]
    As in the proof of Theorem \ref{T1}, our measure $d\mu_w$ satisfies Curvature-Dimension condition $CD(1+K,\infty)$. Therefore, by \cite[Proposition 5.7.1]{BGL14}, we obtain a Logarithmic Sobolev inequality with constant $\frac{1}{1+K_w}$:
\begin{equation*}
    \text{Ent}_{d\mu_w}(f^2) \leq \frac{2}{1+K_w}\int_\Sigma |\nabla f|^2 d\mu_w .
\end{equation*}
Now we will show that for $w$ partial, the equality is achieved. Without loss of generality, suppose that $w$ is independent of $x_n$ and consider the function $f = e^{bx_n}$. Then $K_w=0$ and we can compute:
\begin{equation*}
    \int_\Sigma f^2 d\mu_w = \frac{1}{\int_\Sigma w(x)e^{-\frac{|x|^2}{2}} dx}\int_\Sigma e^{2bx_n} w(x) e^{-\frac{|x|^2}{2}}dx = \frac{\int_{-\infty}^\infty e^{2bx_n - \frac{1}{2}x_n^2}dx_n}{\int_{-\infty}^\infty e^{-\frac{x_n^2}{2}}dx} = e^{2b^2}.
\end{equation*}
Then 
\begin{align*}
    \text{Ent}_{d\mu_w}(f^2) &= \int_\Sigma f^2 \log f^2 d\mu_w - \int_\Sigma f^2 d\mu_w\log\left(\int_\Sigma f^2 d\mu_w\right) \\
    &= \int_\Sigma e^{2bx_n} (2bx_n) d\mu_w - e^{2b^2} (2b^2) \\
    &= 2b(2be^{2b^2}) - 2b^2 e^{2b^2} \\
    &= 2b^2 e^{2b^2}.
\end{align*}
On the other hand 
\begin{align*}
    \int_\Sigma |\nabla f|^2 d\mu_w = \int_\Sigma b^2 e^{2bx_n}d\mu_w = b^2 \frac{\int_{-\infty}^\infty e^{2bx_n - \frac{1}{2}x_n^2}dx_n}{\int_{-\infty}^\infty e^{-\frac{x_n^2}{2}}dx} = b^2e^{2b^2}.
\end{align*}
So we have that 
$$\text{Ent}_{d\mu_w}(f^2) = 2 \int_\Sigma |\nabla f|^2 d\mu_w.$$
\end{proof}
Next we will show how this implies a logarithmic Sobolev inequality for log concave, homogeneous measure $d\nu = w(x) dx$.

\begin{proof}[Proof of Theorem \ref{T6}]
    Recall that $d\mu_w = C_w w(x) e^{-\frac{|x|^2}{2}}dx$ where $C_w = \left(\int_\Sigma w(x) e^{-\frac{|x|^2}{2}}dx\right)^{-1}$. Define $h(x) = \sqrt{C_w} e^{-\frac{|x|^2}{4}}$ so that $h^2 d\nu = d\mu_w$. Since $w$ is log-concave, by Theorem \ref{T5}, the homogeneous Gaussian LSI~\eqref{LSIHG} holds for $K_w = 0$. Then we can take $F = \frac{f}{h}$. On the left hand side we get 
    \begin{align*}
        \ent(F^2) &= \int_\Sigma F^2 \log F^2d\mu_w- \int_\Sigma F^2 d\mu_w\log\left(\int_\Sigma F^2 d\mu_w\right)\\
        &= \int_\Sigma  f^2 \log\left(\frac{f^2}{h^2}\right)d\nu - \int_\Sigma f^2 d\nu \log\left(\int_\Sigma f^2 d\nu\right) \\
        &= \text{Ent}_{d\nu}(f^2) - \int_\Sigma f^2 \log h^2 d\nu.
    \end{align*}
    On the right hand side we get 
    \begin{align*}
        \int_\Sigma |\nabla F|^2 d\mu_w&= \int_\Sigma 
        \left|\frac{\nabla f - f\nabla \log h}{h}\right|^2 d\mu_w\\
        &= \int_\Sigma \left|\nabla f +\frac{xf}{2}\right|^2 d\nu \\
        &= \int_\Sigma |\nabla f|^2 d\nu + \int_\Sigma \nabla f\cdot xf d\nu +\int_\Sigma \frac{1}{4}|x|^2 f^2 d\nu.
    \end{align*}
    The middle term we integrate by parts to get 
    \begin{align*}
        \int_\Sigma \nabla f\cdot xf d\nu &= - \frac{1}{2}\int_\Sigma f^2 \text{div}(xw)dx \\
        &= - \frac{1}{2}\int_\Sigma f^2 (n w+ x\cdot\nabla w)dx\\
        &= -\frac{n+\alpha}{2}\int_\Sigma f^2 d\nu.
    \end{align*}
    Putting everything together and simplifying: 
    \begin{align*}
        \text{Ent}_{d\nu}(f^2) \leq 2\bigg(\int_\Sigma |\nabla f|^2 d\nu -\frac{n+\alpha}{2}\int_\Sigma f^2 d\nu &+\int_\Sigma \frac{1}{4}|x|^2 f^2 d\nu\bigg) \\
        &+ \int_\Sigma f^2 \log C_w d\nu - \int_\Sigma \frac{1}{2}|x|^2f^2 d\nu .
    \end{align*}
    To simplify the notation let 
    \begin{align*}
        A = \int_\Sigma |\nabla f|^2 d\nu, \qquad B = \int_\Sigma f^2 d\nu, \qquad D= \int_\Sigma |x|^2 f^2 d\nu.
    \end{align*}
    Then the inequality reads
    \begin{align*}
        \text{Ent}_\nu(f^2) &\leq 2\bigg(A -\frac{n+\alpha}{2}B +\frac{1}{4}D\bigg) + B\log C_w - \frac{1}{2}D \\
        &= 2A + \left(\log C_w - (n+\alpha)\right)B + \left(\frac{1}{2}-\frac{1}{2}\right)D.
    \end{align*}
    The last term is identically zero. Thus we get
    \begin{align}
        \text{Ent}_\nu(f^2) 
        &\leq 2\int_\Sigma |\nabla f|^2 d\nu + \left(\log C_w - (n+\alpha)\right)\int_\Sigma f^2 d\nu  .\label{helper6.1}%
    \end{align}
      We want make a change of function $f_\lambda(x) = \lambda^{\frac{n+\alpha}{2}}f(\lambda x)$. The norm $|| \cdot||_{2,\nu}$ is preserved under this rescaling. 
      \begin{align*}
        ||f_\lambda||^2_{2,\nu} &= \int_\Sigma \lambda^{n+\alpha}f(\lambda x) w(x) dx = \int_\Sigma \lambda^{n+\alpha}f(y) w\left(\frac{y}{\lambda}\right) \lambda^{-n} dy \\
        &= \int_\Sigma \lambda^{n+\alpha}\lambda^{-n-\alpha}f(y)w(y)dy = ||f||^2_{2,\nu}.
    \end{align*}
    \begin{align*}
        \int_\Sigma |\nabla f_\lambda|^2 d\nu &= \int_\Sigma \lambda^{n+\alpha }\left|(\nabla f)(\lambda x)\right|^2 d\nu = \lambda^{n+\alpha }\int_\Sigma |\lambda\nabla f(\lambda x)|^2 d\nu \\
        &= \lambda^{n+\alpha + 2}\int_\Sigma |\nabla f(y)|^2 \lambda^{-n-\alpha} w(y)dy \\
        &= \lambda^2 \int_\Sigma |\nabla f|^2 d\nu.
    \end{align*}
    \begin{align*}
        \text{Ent}_\nu(f_\lambda^2) &= \int_\Sigma f_\lambda^2 \log f_\lambda^2 d\nu - \int_\Sigma f_\lambda^2 d\nu \log\left(\int_\Sigma f_\lambda^2 d\nu\right) \\
        &= \int_\Sigma f_\lambda^2 \log \lambda^{n+\alpha} + f_\lambda^2 \log f^2(\lambda x) d\nu - \int_\Sigma f^2 d\nu \log\left(\int_\Sigma f^2 d\nu\right) \\
        &= \log\lambda^{n+\alpha}\int_\Sigma \lambda^{n+\alpha}f^2(\lambda x)w(x) dx + \int_\Sigma \lambda^{n+\alpha} f^2(\lambda x) \log f^2(\lambda x) w(x) dx \\
        &-\int_\Sigma f^2 d\nu \log\left(\int_\Sigma f^2 d\nu\right) \\
        &= \text{Ent}_{d\nu}(f^2) + \log\lambda^{n+\alpha} \int_\Sigma f^2 d\nu .
    \end{align*}
    Rewriting inequality \ref{helper6.1} for $f_\lambda$ we have
    \begin{align}
        \text{Ent}_{d\nu}(f^2) + \log\lambda^{n+\alpha} \int_\Sigma  f^2 d\nu  \leq 2\lambda^2 \int_\Sigma |\nabla f|^2 d\nu + \left(\log C_w - (n+\alpha)\right)\int_\Sigma  f^2 d\nu. \label{helper6.2}%
    \end{align}
    Consider the function $R(\lambda)$ 
    \begin{equation*}
        R(\lambda) = 2\lambda^2 \int_\Sigma |\nabla f|^2 d\nu + \left(\log C_w - (n+\alpha)\right)\int_\Sigma  f^2 d\nu -  \log\lambda^{n+\alpha} \int_\Sigma  f^2 d\nu.
    \end{equation*}
    Then $R$ achieves a minimum at 
    \begin{equation*}
        \lambda^2 = \frac{(n+\alpha)}{4} \frac{\int_\Sigma  f^2 d\nu}{\int_\Sigma  |\nabla f|^2 d\nu}.
    \end{equation*}
    Optimizing inequality \ref{helper6.2} in $\lambda$ we get 
     \begin{align*}
        \text{Ent}_{d\nu}(f^2) &\leq 2\int_\Sigma  |\nabla f|^2 d\nu\frac{(n+\alpha)}{4} \frac{\int_\Sigma  f^2 d\nu}{\int_\Sigma  |\nabla f|^2 d\nu} + \left(\log C_w - (n+\alpha)\right)\int_\Sigma  f^2 d\nu \\
        &- \frac{n+\alpha}{2} \int_\Sigma  f^2 d\nu\log\frac{(n+\alpha)}{4} \frac{\int_\Sigma  f^2 d\nu}{\int_\Sigma  |\nabla f|^2 d\nu}.
    \end{align*}
    Simplifying 
    \begin{align*}
         \text{Ent}_{d\nu}(f^2) &\leq \frac{n+\alpha}{2}\int_\Sigma  f^2 d\nu +  \left(\log C_w - (n+\alpha)\right)\int_\Sigma  f^2 d\nu \\
        &- \frac{n+\alpha}{2} \int_\Sigma  f^2 d\nu\log\frac{(n+\alpha)}{4} \frac{\int_\Sigma  f^2 d\nu}{\int_\Sigma  |\nabla f|^2 d\nu} \\
        &= \frac{n+\alpha}{2}\int_\Sigma  f^2 d\nu \left(1+ \log C_w^{\frac{2}{n+\alpha}} - 2 - \log\frac{(n+\alpha)}{4} \frac{\int_\Sigma  f^2 d\nu}{\int_\Sigma  |\nabla f|^2 d\nu}\right) \\
        &=  \frac{n+\alpha}{2}\int_\Sigma  f^2 d\nu \left(\log e^{1-2}+ \log C_w^{\frac{2}{n+\alpha}} - \log\frac{(n+\alpha)}{4} \frac{\int_\Sigma  f^2 d\nu}{\int_\Sigma  |\nabla f|^2 d\nu}\right) \\
        &= \frac{n+\alpha}{2}\int_\Sigma  f^2 d\nu \log \left(\frac{4e^{-1}  C_w^{\frac{2}{n+\alpha}} }{(n+\alpha)} \frac{\int_\Sigma  |\nabla f|^2 d\nu}{\int_\Sigma  f^2 d\nu}\right) \\
        &= \frac{n+\alpha}{2}\int_\Sigma f^2 d\nu \ \log \left(C_{LSIH}\frac{\int_\Sigma |\nabla f|^2 d\nu}{\int_\Sigma f^2 d\nu}\right).
    \end{align*}
   Next we show that the constant is sharp. First, we consider the case when $A = 1$ and $f = e^{-\frac{|x|^2}{4}}$. Then 
    \begin{equation*}
        \int_\Sigma f^2 d\nu = \int_\Sigma f^2 w(x) dx = \int_\Sigma w(x)e^{-\frac{|x|^2}{2} }dx = C_w^{-1}.
    \end{equation*}
    On the left hand side we have
    \begin{align*}
        \text{Ent}_{d\nu}(f^2) &= \int_\Sigma f^2 \log f^2 d\nu - \int_\Sigma f^2 d\nu \log \left(\int_\Sigma f^2 d\nu\right) \\
        &= \int_\Sigma e^{-\frac{|x|^2}{2}} \left(-\frac{|x|^2}{2}\right) d\nu - \frac{1}{C_w}\log \frac{1}{C_w} \\
        &= -\frac{1}{2}\int_\Sigma |x|^2  e^{-\frac{|x|^2}{2}} d\nu + \frac{\log C_w}{C_w}.
    \end{align*}
    To compute this integral, let $V(x) = x e^{-\frac{|x|^2}{2}}  w(x)$. Then 
    \begin{align*}
        \text{div}(V) &= ne^{-\frac{|x|^2}{2}}  w(x) +\left(-xe^{-\frac{|x|^2}{2}}\right)\cdot xw(x) + xe^{-\frac{|x|^2}{2}}\cdot\nabla w \\
        &= (n+\alpha)e^{-\frac{|x|^2}{2}}w(x) - |x|^2 e^{-\frac{|x|^2}{2}}w(x).
    \end{align*}
    From the divergence theorem and decay at $\infty$ we have 
    $$\int_\Sigma \text{div}(V) dx = 0.$$
    So $$\int_\Sigma |x|^2 e^{-\frac{|x|^2}{2}} w(x) dx = (n+\alpha)\int_\Sigma e^{-\frac{|x|^2}{2}} w(x) dx = \frac{n+\alpha}{C_w}.$$
    Thus $$\text{Ent}_{d\nu}(f^2) = -\frac{n+\alpha}{2C_w} + \frac{\log C_w}{C_w}.$$
    $$\nabla f = -\frac{x}{2} e^{-\frac{|x|^2}{4}}.$$
    And 
    $$\int_\Sigma |\nabla f|^2 d\nu = \frac{1}{4}\int_\Sigma |x|^2e^{-\frac{|x|^2}{2}}d\nu = \frac{n+\alpha}{4C_w}. $$
    $$\frac{\int_\Sigma |\nabla f|^2 d\nu }{\int_\Sigma f^2 d\nu} = \frac{(n+\alpha )/4C}{\frac{1}{C_w}} = \frac{n+\alpha}{4}.$$
    So on the right hand side we have 
    \begin{align*}
        RHS &=\frac{n+\alpha}{2} \int_\Sigma f^2 d\nu \  \log \left(\frac{4}{e(n+\alpha)}C_w^{\frac{2}{n+\alpha}} \frac{\int_\Sigma |\nabla f|^2 d\nu}{\int_\Sigma f^2 d\nu} \right) \\
        &= \frac{n+\alpha}{2} \frac{1}{C_w} \log\left(\frac{4}{e(n+\alpha)}C_w^{\frac{2}{n+\alpha}} \frac{n+\alpha}{4}\right) \\
        &= \frac{n+\alpha}{2} \frac{1}{C_w} \log\left(\frac{C_w^{\frac{2}{n+\alpha}}}{e} \right) \\
        &= \frac{n+\alpha}{2C_w} \left(\frac{2}{n+\alpha}\log C_w - 1\right) \\
        &= \frac{\log C_w}{C_w}- \frac{n+\alpha}{2C_w}\\
        &= LHS.
    \end{align*}
    Now for the general case $A \neq 1$ we just note that the entropy term 
    \begin{equation*}
        \ent\left(g^2\right) = \ent\left((Af)^2\right) = \ent(A^2 f^2) = A^2 \ent(f^2).
    \end{equation*}
    And on the right hand side we have 
    $$A^2 \ \frac{n+\alpha}{2} \int_\Sigma f^2 d\nu \log\left(C_{LSIH} \frac{A^2 \int_\Sigma |\nabla f|^2 d\nu}{A^2 \int_\Sigma f^2 d\nu }\right).$$
    So we still have equality for $g = Af$. 
\end{proof}

Next we will prove the equivalence of the homogeneous LSI and the Gaussian homogeneous LSI. 
\begin{proof}[Proof of corollary \ref{C2}]
    Theorem \ref{T6} gives us one direction. To show the other direction we use the following inequality. For $s,t>0$, $$\log t \leq st -\log s -1. $$
    Let $$t = \frac{4}{e(n+\alpha)}C_w^{\frac{2}{n+\alpha}} \frac{\int_\Sigma |\nabla f|^2 d\nu}{\int_\Sigma f^2 d\nu}.$$
    and take $s =  \frac{e}{C_w^{\frac{2}{n+\alpha}}}$. Then 
    \begin{align*}
        \text{Ent}_{d\nu}(f^2)  &\leq \frac{n+\alpha}{2}\int_\Sigma f^2 d\nu \  \log \left(\frac{4}{e(n+\alpha)}C_w^{\frac{2}{n+\alpha}} \frac{\int_\Sigma |\nabla f|^2 d\nu}{\int_\Sigma f^2 d\nu} \right) \\
        &\leq \frac{n+\alpha}{2}\int_\Sigma f^2 d\nu \  \left(\frac{e}{C_w^{\frac{2}{n+\alpha}}} \frac{4}{e(n+\alpha)}C_w^{\frac{2}{n+\alpha}} \frac{\int_\Sigma |\nabla f|^2 d\nu}{\int_\Sigma f^2 d\nu} - \log\frac{e}{C_w^{\frac{2}{n+\alpha}}} - 1 \right) \\
        &= \frac{n+\alpha}{2}\int_\Sigma f^2 d\nu \ \left(\frac{4}{n+\alpha}\frac{\int_\Sigma |\nabla f|^2 d\nu}{\int_\Sigma f^2 d\nu} 
        - \log e + \frac{2}{n+\alpha}\log C_w - 1\right) \\
        &= 2\int_\Sigma |\nabla f|^2 d\nu - (n+\alpha) \int_\Sigma f^2 d\nu \  + \log C_w \int_\Sigma f^2 d\nu. \ 
    \end{align*}
    Adding the term $\frac{1}{2}\int_\Sigma |x|^2 f^2d\nu$ to both sides we can rewrite this as 
    \begin{equation*}
        \text{Ent}_{d\nu}(f^2)  + \frac{1}{2}\int_\Sigma |x|^2 f^2d\nu - \int_\Sigma \log C_w f^2d \nu \leq 2\left(\int_\Sigma |\nabla f|^2 d\nu - (n+\alpha)\int_\Sigma f^2 d\nu +\frac{1}{4}\int_\Sigma |x|^2 f^2 d\nu\right).
    \end{equation*}
    After the substitution $f = Fh$ where $h =  \sqrt{C_w}e^{-\frac{|x|^2}{4}}$ we get the result
    $$ \text{Ent}_{d\mu_w}(F^2)\leq 2\int_\Sigma |\nabla F|^2 d\mu_w.$$
\end{proof}

\end{document}